\documentclass[12pt]{amsart}
\usepackage{graphicx,psfrag,epsfig}
\usepackage{amsmath}
\usepackage{amsthm}
\usepackage[top=2.5cm, bottom=2.5cm]{geometry}
\usepackage{amsfonts}
\usepackage{amssymb}
\usepackage{color}

\newtheorem{lemma}{Lemma}[section]
\newtheorem{theorem}[lemma]{Theorem}
\newtheorem{corollary}[lemma]{Corollary}
\newtheorem{proposition}[lemma]{Proposition}

\theoremstyle{definition}

\newtheorem{remark}[lemma]{Remark}

\newtheorem{definition}{Definition}

\def\supp{{\mathrm{supp}}}

{

  \bigskip\bigskip%
}

\newcommand{\SL}{\mathrm{SL}(m,\mathbb{R})}
\newcommand{\Var}{\mathrm{Var}}
\begin{document}

%
%


\title[Products of non-stationary Markov-dependent matrices]
{Exponential growth of products of non-stationary Markov-dependent matrices}
\author{I. Goldsheid}

\address{Ilya Goldsheid: School of Mathematical Sciences\\
Queen Mary University of London\\
London E1 4NS \\
 Great Britain\\
email: I.Goldsheid@qmul.ac.uk}


\begin{abstract}
Let $(\xi_j)_{j\ge1} $, be a non-stationary Markov chain with phase space $X$ and let
$\mathfrak{g}_j:\,X\mapsto\SL$ be a sequence of functions on $X$ with values in the unimodular group.
Set $g_j=\mathfrak{g}_j(\xi_j)$ and denote by $S_n=g_n\ldots g_1$, the product of the matrices $g_j$.
 We provide sufficient conditions for exponential growth of the norm $\|S_n\|$ when the Markov chain is not supposed
to be stationary. This generalizes the classical theorem of Furstenberg on the exponential growth of products of independent identically
distributed matrices as well as its extension by Virtser to products of stationary Markov-dependent matrices.

\end{abstract}

\maketitle

\section{Introduction}
Let $(g_n)_{n\ge1}$ be a sequence of matrices, $g_n\in \mathrm{SL}(m,\mathbb{R})$, and set
\[
S_n=g_n...g_1.
\]
 
In the seminal 1963 paper \cite{F}, H. Furstenberg proved the following theorem. (All relevant definitions can be
found in section \ref{SecMatrices}.)
\begin{theorem}[Furstenberg, \cite{F}]\label{ThmF}
Suppose that:\newline
$\mathrm{(a)}$ $(g_n)_{n\ge1}$ is a sequence of independent identically distributed 
i.i.d. random matrices 
satisfying $\int_{\SL}\log\| g\|\;d\nu(g)<\infty$. \newline
$\mathrm{(b)}$ The group $\mathbb{G}_\nu$ generated by the support of $\nu$ does not preserve any probability measure on the unit sphere
$\mathcal{S}$ in $\mathbb{R}^m$.

Then the following limit (called the top Lyapunov exponent of the product $S_n$) exists with probability 1 and is strictly positive:
\begin{equation}\label{M}
\lim_{n\to\infty} \frac{1}n \ln\|S_n\|=\lambda>0.
\end{equation}
\end{theorem}
The existence of the limit in \eqref{M} was first proved by Furstenberg and Kesten in \cite{FK} for a 
stationary sequence $(g_n)_{n\ge1}$. 
The main statement of Theorem \ref{ThmF} is the strict positivity of $\lambda$.

In the late 1970s A. Virtser \cite{V} 
extended this result to products
of stationary Markov-dependent sequences of matrices by which we mean that $g_n=\mathfrak{g}(\xi_n)$, where
$\xi=(\xi_n)_{n\ge1} $ is a stationary Markov chain with a phase space $X$ and
$\mathfrak{g}:\,X\mapsto\SL$ is a `good' function on $X$. 
The Markov chain $\xi$ in \cite{V} is supposed
to satisfy the condition $\|K^0\|<1$, where $K^0$ is the restriction of the transition operator of the chain to
the subspace of functions on $X$ orthogonal to constants 
(the exact statement of Virtser's theorem is given in section \ref{SecExample3}). The other conditions in \cite{V} are
as in Furstenberg's theorem.

The paper \cite{V} was followed by a number of works of which we mention Royer \cite{R}, Guivarc'h \cite{Gu}, Ledrappier \cite{L} where the stationarity 
of the chain still plays a crucial role.

The main goal of this work is to extend Virtser's result (and thus also Furstenberg's result) to the product of
Markov-dependent matrices in the case when the underlying Markov chain is not supposed to be stationary.
Obviously, in this case the limit in \eqref{M} may not exist. What we shall show is that there is $\lambda>0$
such that with probability 1 $\liminf_{n\to\infty} \frac{1}n \ln\|S_n\|\ge\lambda$. This will be done under 
very mild conditions. Namely:

\begin{itemize}
\item We remove completely the requirement for the Markov chain  to be stationary. In contrast,
stationarity is crucial for the technique used in majority of previous work.

\item The functions defining the matrices $g_n$ may be time-dependent: $g_n=\mathfrak{g}_n(\xi_n)$,
where $\mathfrak{g}_n:\,X\mapsto\SL$. Moreover, they may themselves be random.

\item No moment condition is imposed on the distributions of matrices (and thus the case $\lambda=+\infty$
is not excluded).

\item The conditions on the transition operators and the supports of the distributions of $g_n$ are supposed to 
be satisfied only for a subsequence of indices of positive density.


\end{itemize}
The price we pay is that the group $\mathbb{G}_\nu$ appearing in assumption (b) of Furstenberg's 
theorem has to be replaced by a smaller group $G_\nu$ generated by  all products $g\tilde{g}^{-1}$, where $g$ and $\tilde{g}$
are  from the support of $\nu$ (in the non-stationary case, these groups depend on $n$).
This means that in the stationary case our requirement is, formally speaking,
more restrictive than the one in \cite{V}. However, we show in Section \ref{SecExample3} that 
Virtser's theorem can be deduced from our result.

%
%

Several important technical lemmas in the present work and in \cite{V} are similar.
We state these lemmas in the form which is convenient for us and we prove them to make this paper
self-contained.
The main innovation with respect to \cite{V} 
is that we manage to upgrade the estimate on the spectral radius of certain operators to an estimate 
on the norm of a product of two operators of the same kind. 
This upgrade enables all the generalizations described above.



\subsection{Motivation}

First of all, the exponential growth is a fundamental property of products $S_n$ and the task of extending it to wider
classes of products is important in its own right.

Here are several examples of problems whose solutions depend, to a large extent, on the
possibility to control the top Lyapunov exponent of the product $S_n$.
The first of these problems will be addressed in this work. The other two problems will be discussed elsewhere.

1. Given a stationary sequence of matrices $g_j$, consider their perturbations of
the form $a_jg_j$, where $a_j$ is a non-random sequence of matrices,  $a_j\in\mathrm{SL}(m,\mathbb{R})$.
The natural question is: what part of the theory of stationary products can be extended to this non-stationary
case?

In section \ref{SecExample2}, we prove that the exponential growth of the product is preserved for the class of
Markov-dependent matrices (which may not be stationary). In many applications this fact is more important
than the existence of the limit in \eqref{M}. (We remark that, unless $a_n$'s are chosen in some special way,
this limit does not exist.)


2. In the theory of Anderson localization in dimension one, the exponential growth of $S_n$ plays
a crucial role. In particular, it implies the existence
with probability 1 of a random vector $u\in\mathbb{R}^2$ such that $\limsup_{n\to\infty}n^{-1}\ln\|S_nu\|<0$.
Extending this theory to models with non-stationary potential is an important problem
and this work is a step in this direction. In the case of  non-stationary potential with independent
entries a solution to this problem was announced in \cite{GorKl}.

3. One of the central questions in the general theory of products of independent random matrices is the existence of distinct
Lyapunov exponents (see e.g. \cite{Viana}). It turns out that this question can be reduced to the question about the growth of
the top Lyapunov exponent of a product of Markov-dependent matrices.




\subsection{Some history: products of independent non-identically distributed matrices}
There is extensive literature studying different aspects of the theory of Lyapunov exponents for
products of a stationary sequence of matrices.  We refer the interested reader
to relatively recent books \cite{Viana} 
and \cite{BQ} and references therein. 

In contrast, there are few papers dealing with products of non-identically distributed matrices
most of which arise in the context of the spectral theory of random Schr\"odinger operators.
We are aware of the following articles.

Works \cite{DSS}, \cite{Simon}, \cite{KLS} 
deal with matrices arising in the theory of localization for Anderson model in dimension one with a potential 
decaying at infinity. These matrices are of the form
\begin{equation}\label{M1}
g_n=\left( \begin{array}{cc}
 a_nq_n & -1 \\
1 & 0
\end{array}\right),
\end{equation}
where $q_n$ are i.i.d. random variables and the (deterministic) sequence $(a_n\in\mathbb{R})_{n\ge1}$ satisfies
$C_1|n|^{-\alpha}<|a_n|<C_2|n|^{-\alpha}$, where $n\not=0$ and $C_1,\ C_2,\ \alpha$ are some positive constants.
We note that for any $\alpha>0$ the growth of the norm $\|S_n\|$  is at most sub-exponential.
%

The work \cite{W1} by Shubin-Vakilian-Wolff  provides constructive estimates for the norm of an operator
which is the average of a certain representation of $\mathrm{SL}(2,\mathbb{R})$, where the average is computed over
the distribution of the matrices.
This result implies a constructive estimate for the exponential growth of products of
matrices \eqref{M1} with $a_n=1$. 
With a bit of additional work, it is possible to extend this result to the case
of independent non-identically distributed matrices of this form.
Formally speaking, the latter has not been explicitly stated in \cite{W1} but it seems plausible that the authors
were aware of this fact (see comments in \cite[page 943]{W1}).

In \cite{K}, Y. Kifer proved a series of results concerned with different aspects of the theory of products of random matrices
whose distributions form a stationary process. In particular, he proves under certain conditions the strict positivity
of the top Laypunov exponent. The stationarity condition is crucial for the technique used in \cite{K} but is not satisfied
for products of matrices considered in the present paper.

Recently,  A. Gorodetski and V. Kleptsyn announced \cite{GorKl} a proof of exponential growth of a product of
independent non-identically distributed matrices under conditions similar to those stated in Theorem \ref{ThmIndep} of this paper.
For the case of $2\times 2$ matrices, \cite{GorKl} contains additional results on the Law of Large Numbers and
Large Deviations for such products.


\textbf{Acknowledgement.} The author would like to thank A. Sodin for patiently reading several versions of this paper
and for constructive critique and useful suggestions. 

\subsection{Organization of the paper} In section \ref{SecMR} we recall some well known definitions
and introduce the related notation in the form which is best suited for what follows; we then state the main results
(Theorems \ref{Thm1} and \ref{Thm4}) 
and provide some comments on them. The applications of the main results are considered in section \ref{SecExamples};
in particular, example 3 (section \ref{SecExample3}) explains how to deduce Virtser's theorem
from Theorem \ref{Thm1}.
In section \ref{SecIndep} we prove a particular case of
Theorem \ref{Thm1}. Namely, Theorem \ref{ThmIndep} considers the case of independent matrices.
There are several reasons for that. First of all, products of independent matrices form
a very important subclass in the theory of products of random matrices which deserves
a separate consideration. Secondly, this allows us to
explain some of the ideas in the case which is less technical and therefore more transparent
than the general case. Finally, the proof in the Markov-dependent case makes use of Lemma \ref{Thm2} which is
the main technical result needed for the independent case.
In section \ref{SecMarkov} we introduce the technique which allows us to treat the
products of Markov-dependent matrices and prove Theorem \ref{Thm1}. The main parts of the proofs 
in the Markov-dependent and in the independent case differ significantly and this difference 
does not seem to be easily predictable (see Remark \ref{RmkMain}). In section 6, our second main result
(Theorem \ref{Thm4}) is derived from Lemma \ref{LemmaParticCase} which, in turn, is an extension of 
Theorem \ref{Thm1}. Appendix contains two elementary lemmas which we use in the main text of the paper.

\subsection{Some notation and conventions} The following notation is used throughout the paper.

$\mathcal{S}$ is the unit sphere in $\mathbb{R}^m$ and $u\in \mathcal{S}$ is a unit vector.
We write $\int_\mathcal{S}f(u)du$ for the integral over the uniform distribution on $\mathcal{S}$.

$\xi$ and $(\xi_n)_{ n\ge1}$ denote the same Markov chain. A similar convention applies to all Markov
chains which are introduced in the paper, such as $\tilde{\xi}=(\tilde{\xi}_n)_{ n\ge1}$, etc .

$X$ is the phase space of $\xi$. The elements of $X$ are denoted $x$, $y$, $x_i$, $y_j$, etc.

The term measure always means probability measure.

The notation $\|\cdot\|$ is mainly used for the norms of
vectors and matrices; in those cases when it is used for norms of functions or operators, its exact
meaning is always obvious from the context.

If $f$ belongs to a space of functions $\mathbb{H}_n$, we write $\|f\|_{\mathbb{H}_n}$ for the norm of $f$
when it is important to emphasize that $f\in\mathbb{H}_n$ and that the norm is the one with
which $\mathbb{H}_n$ is equipped.

\section{Main results}\label{SecMR}
\subsection{The setup}

\subsubsection{The Markov chain.}\label{SecDefMC}
Let $(X,\mathcal{B})$ be a measurable space (with $\mathcal{B}$ being the sigma-algebra of measurable subsets of the
set $X$). Consider a Markov chain $\xi=(\xi_n)_{n\ge 1}$,
with the phase space $X$ and the initial distribution $\mu_1$.
For any $B\in\mathcal{B}$, set
\[
k_n(x,B)=\mathbb{P}(\xi_{n+1}\in B\, |\,\xi_n=x).
\]
We write $k_n(x,dy)$ for the corresponding transition kernels of the chain $\xi$.

Let $\mu_n$ be the distribution of $\xi_n$. As usual, for $n\ge 2$ and $B\in\mathcal{B}$ we have
\begin{equation}\label{1.H}
\mu_{n}(B)=\mathbb{P}(\xi_{n}\in B)=\int_X \mu_{n-1}(dx)k_{n-1}(x,B).
\end{equation}
We thus have a sequence of `Markov related' measure spaces $(X,\mathcal{B},\mu_n)$. Denote by $H_n$ the
Hilbert space of $\mu_n$-square integrable real valued functions,
\begin{equation}\label{2.H}
H_n=\{f:X\mapsto\mathbb{R},\ \int_X|f(x)|^2\mu_n(dx)<\infty\}
\end{equation}
with the standard inner product: if $f,\, h\in H_n$ then $\left<f,h\right>_{H_n}=\int_Xf(x)h(x)\mu_n(dx).$
Set
\begin{equation}\label{DefH^0}
 H_n^{0}=\{f\in H_n:\, \int_Xf(x)\mu_n(dx)=0\}.
\end{equation}
The integral with respect to $\mu_n$ will be denote $\mathbb{E}_{n}$ :  $\mathbb{E}_{n}(f)= \int_Xf(x)\mu_n(dx)$.

Let $K_n: H_{n+1}\mapsto H_n$ be the operator defined by
\[
(K_nf)(x)=\int_Xk_n(x,dy)f(y).
\]
We remark that if $f\in H_{n+1}$ then $K_nf\in H_{n}$ which is a standard property of any Markov chain.
Note that the operator $K_n$ `computes' the conditional expectation of $f(\xi_{n+1})$ conditioned on $\xi_n=x$.

Denote by $K_n^{0}$ the restriction of $K_n$ to $H_{n+1}^{0}$. It is easy to see that if $\mathbb{E}_{n+1}(f)=0$ then
$\mathbb{E}_{n}(K_nf)=0$, that is $K_n^{0}:H_{n+1}^{0}\mapsto H_{n}^{0}$.

\smallskip
\subsubsection{The matrices.}\label{SecMatrices}
Let $\mathfrak{g}_n:X\mapsto \mathrm{SL}(m,\mathbb{R})$, $n\ge1$, be a sequence of matrix-valued $\mathcal{B}$-measurable
functions on $X$.
Define a sequence of random matrices $g_j$ by setting $g_j=\mathfrak{g}_j(\xi_j),\ j\ge1$. Let $\nu_j$ be the
distribution of $g_j$, that is for a Borel set $\Gamma\subset \mathrm{SL}(m,\mathbb{R})$ we define
\begin{equation}\label{Def:nu}
\nu_j(\Gamma)=\mathbb{P}(\mathfrak{g}_j(\xi_j)\in\Gamma).
\end{equation}
By $\supp(\nu_j)\subset \mathrm{SL}(m,\mathbb{R})$ we denote the support of $\nu_j$.

Given a distribution $\nu$ on $\mathrm{SL}(m,\mathbb{R})$ we define the group $G_{\nu}$ as follows:
\begin{equation}\label{Def:Gnu}
G_{\nu}=\text{ closed group generated by the set }\{g_1g_2^{-1}\,: \ g_1,\,g_2 \in\supp(\nu) \}.
\end{equation}
By $\mathcal{S}$ we denote the unit sphere in $\mathbb{R}^m$.
\begin{definition} For $g\in \mathrm{SL}(m,\mathbb{R})$ and $u\in\mathcal{S}$ we define $g.u=gu/||gu||$.

The induced action of $g$ on the set of probability measures on $\mathcal{S}$ is defined by $(g\kappa)(B)=\kappa(g.^{-1}B)$,
where $\kappa$ is a probability measure on $\mathcal{S}$ and $B$ is a Borel subset of $\mathcal{S}$.

We say that a probability measure $\kappa$ on $\mathcal{S}$ is preserved by $g$ if $\kappa(B)=(g\kappa)(B)$ for
any Borel $B$.

A group $G$ preserves the measure $\kappa$ on $\mathcal{S}$ if every
$g\in G$ preserves $\kappa$.
\end{definition}

\begin{remark} We can add one more degree of randomness to the way the matrices $g_n$ are defined.
Namely, let $(\Omega, \mathcal{F}, \ell)$ be a probability space and $\mathfrak{g}=(\mathfrak{g}_n)_{n\ge1}$ be a sequence of independent 
random processes on $\Omega$ with `time' parameter $x\in X$ and values in $\SL$: for $x\in X$ and $\omega\in\Omega$ 
the value of the process at time $n$ is $\mathfrak{g}_n(x, \omega)$. We suppose of course that the function $\mathfrak{g}_n(\cdot, \cdot)$
is $\mathcal{B}\times \mathcal{F}$-measurable. Next, suppose that the Markov chain $\xi$ and the sequence $\mathfrak{g}$ are independent
and set $g_n=\mathfrak{g}_n(\xi_n, \omega)$.

This definition of the sequence $(g_n)_{n\ge1}$ provides us with a wider class of products which may be useful in applications.
However, the measures $\nu_n$ are once again defined by \eqref{Def:nu} (up to a natural modification), and because of that
both the statements and the proofs of all main results are identical for this class of processes and for the 
case of the deterministic sequence $(\mathfrak{g}_n)_{n\ge1}$ considered above.
\end{remark}

\subsection{Main results}
The following assumptions and their variations will be used throughout the paper.

I. There is a $c<1$ such that for all $n\ge 1$
\begin{equation}\label{MainCond}
\| K_n^{0}\|\le c.
\end{equation}

II. There is a set $M$ of probability measures on $\SL$ which is compact with respect
to weak convergence and such that:
\begin{equation}\label{MainCond2}
\begin{aligned}
(a)& \text{ all $\nu_n$ belong to $M$,}\\
(b)& \text{ all $\nu\in M$ are such that $G_{\nu}$ does not preserve any measure on $\mathcal{S}$.}
 \end{aligned}
\end{equation}

\begin{theorem}\label{Thm1}{\it
Suppose that assumptions I and II are satisfied. Then there is $\lambda>0$ such that with
probability 1}
\begin{equation}\label{M2}
\liminf_{n\to\infty}\frac1n \ln\|g_n\ldots g_1\|\ge \lambda.
\end{equation}
\end{theorem}
Our next result allows us to relax the assumptions of Theorem \ref{Thm1}: it turns out that \eqref{M2} holds
when I and II are satisfied only for a subsequence of time moments. To state it, we need two more
definitions.
For integers $n\ge1$ and $l\ge1$ denote by  $\nu_{n,l}$ the distribution of the product
$\mathfrak{g}_{n+l-1}(\xi_{n+l-1})...\mathfrak{g}_{n}(\xi_{n})$: for
a Borel set $\Gamma\subset \mathrm{SL}(m,\mathbb{R})$ we set
\[
\nu_{n,l}(\Gamma)=\mathbb{P}(\mathfrak{g}_{n+l-1}(\xi_{n+l-1})...\mathfrak{g}_{n}(\xi_{n})\in\Gamma)\ \text{ and}
\]
\begin{equation}\label{Def:Gnul}
G_{\nu_{n,l}}=\text{ closure of the group generated by the set }
\{g\tilde{g}^{-1}\,:\,g, \tilde{g}\in\supp(\nu_{n,l}) \}.
\end{equation}
We note that $\nu_{n,1}=\nu_n$, $G_{\nu_{n,1}}=G_{\nu_{n}}$.
\begin{theorem}\label{Thm4}
{\it Suppose that there is a sequence of time intervals $[n_j,\,n_j+l_j],\ j\ge1,$ with properties
$n_1\ge1$, $l_j\ge1$, $n_{j+1}\ge n_j+l_j$ and such that:\newline
$\mathrm{(i)}$ the inequalities $\|K_{n_j}^0\|\le c$, $\|K_{n_j+l_j}^0\|\le c$, where $ c<1$, hold for all $j\ge 1$;\newline
$\mathrm{(ii)}$  the distributions $\nu_{n_j+1,l_j}$ belong to a compact set $M$ satisfying the requirement \eqref{MainCond2}(b).

Then there is a (non-random) $\lambda>0$ such that with probability 1}
\begin{equation}\label{M22}
\liminf_{j\to\infty}\frac1j \ln\|g_{n_j+l_j}\ldots g_1\|\ge \lambda
\end{equation}
and $\lambda$ in \eqref{M22} does not depend on the choice of functions
\begin{equation}\label{M24}
\{\mathfrak{g}_i(\cdot)\,:\ i\in \bigcup_{j\ge1}[n_j+l_j+1, n_{j+1}]\, \}.
\end{equation}
\end{theorem}
Theorem \ref{Thm1} is a particular case of Theorem \ref{Thm4} with $n_j=j$ and $l_j=1$. However, we shall see
in section \ref{SecThm4} that the proof of Theorem \ref{Thm4} will be reduced to the proof of Theorem \ref{Thm1}.

Here is one more useful particular case of Theorem \ref{Thm4}.
\begin{corollary}\label{Thm3}{\it
Suppose that assumption I is satisfied and there is $k\ge1$ such that
all distributions $\nu_{nk+1,k},\ n\ge0,$ belong to a compact set $M$ satisfying \eqref{MainCond2}(b).

Then there is a (non-random) $\lambda>0$ such that with probability 1}
\[
\liminf_{n\to\infty}\frac1n \ln\|g_n\ldots g_1\|\ge \lambda.
\]
\end{corollary}
\textbf{Remarks.}\begin{enumerate}
\item[1.] If in Theorem \ref{Thm4} $n_j+l_j= n_{j+1}$ then, by convention, the interval
$[n_j+l_j+1, n_{j+1}]$ in \eqref{M24} is empty .

\item[2.] If $\xi$ is a finite Markov chain with $X=\{1,...,r\}$ then it is easy to see that \eqref{MainCond}
holds if there is $\delta>0$ such that $\mathbb{P}(\xi_{n+1}=j\big|\xi_n=i )\ge\delta$ for all $n\ge1$
and all $i,j\in X$. Similarly, \eqref{MainCond} is satisfied if $X$ is a compact metric space and
$k_n(x,dy)=\bar{k}_n(x,y)dy$, where $dy$ is a measure on $X$ and $(\bar{k}_n(x,y))_{n\ge1}$ is a sequens of equicontinuous
functions of $(x,y)$ such that $\bar{k}_n(x,y)\ge\delta>0$.

In the case when $\xi$ is a stationary ergodic Markov chain with finite phase space $X$,
the necessary and sufficient condition for \eqref{MainCond} can be easily established (see e.g. \cite{V}).

\item[3.] If matrices $g_n$ are independent then
\begin{equation}\label{Appl1}
G_{\nu_{n+1,l-1}}\subset G_{\nu_{n,l}}\ \text{ for any } n\ge 1,\ l\ge 2.
\end{equation}
Indeed, definition \eqref{Def:Gnul} of $g$ and $\tilde{g}$ implies that if $g, \tilde{g}\in\supp(\nu_{n+1,l-1})$ then
$gg', \tilde{g}\tilde{g}'\in\supp(\nu_{n,l})$, where $g', \tilde{g}'\in\supp(\nu_{n})$. Due to independence, we can
choose $g'=\tilde{g}'$ (while $g$ and $\tilde{g}$ remain fixed) and hence if $g\tilde{g}^{-1}\in G_{\nu_{n+1,l-1}}$
then $g\tilde{g}^{-1}=gg'(\tilde{g}g')^{-1}\in G_{\nu_{n,l}}$.

In applications, it may happen that the group $G_{\nu_{n+1,l-1}}$ does preserve some probability measure on the unit sphere
but the larger group $G_{\nu_{n,l}}$ doesn't.

In the case of Markov-dependent matrices, the same argument proves \eqref{Appl1} if
the support of the conditional distribution of $\{\mathfrak{g}_{n+l+1}(\xi_{n+l+1})\big|\xi_{n+l}\}$
does not depend on $\xi_{n+l}$. However, if this condition is not satisfied, \eqref{Appl1} may fail.

The example considered in section \ref{SecExample1} shows that in order to check that Condition II in above theorems
is satisfied, it may be sufficient to establish that only a subgroup of $G_{\nu_{n+1,l-1}}$ belongs to $ G_{\nu_{n,l}} $.

\end{enumerate}

\section{Some applications of the main results}\label{SecExamples}

\subsection{The classical matrices \eqref{M1} in the Markov setting}\label{SecExample1}

The example of the product of matrices \eqref{M1} is the particularly well known one (see
Introduction). Throughout this section, we suppose that in \eqref{M1} all $a_n=1$
and that \newline
(a) $q_n=\mathfrak{g}_n(\xi_n)$, where $(\xi_n)_{n\ge1}$ is a Markov chain satisfying assumption I\newline
(b) there are $\delta>0,\ C>0$ such that $\Var(q_n)\ge\delta$ and $|q_n|\le C$ for all $n\ge1$.

Let $q$ and $\tilde{q}$ denote two distinct points from the support of $q_n$  and denote by $g$ and $\tilde{g}$
the matrices corresponding to $q$ and $\tilde{q}$ respectively. Then
$
g\tilde{g}^{-1}=\left( \begin{array}{cc}
1 & q-\tilde{q} \\
0 & 1
\end{array}\right)
$
and hence $G_{\nu_n}$ is a subgroup of the group of upper triangular matrices. Since $g\tilde{g}^{-1}e=e$,
where $e=(1,\,0)^T$, the action of $G_{\nu_n}$ on the unit sphere $\mathcal{S}$ preserves any probability measure
supported by the set $\{e,-e\}\subset \mathcal{S}$. Hence the exponential growth of $S_n$ doesn't follow from
Theorem \ref{Thm1} since assumption II of this theorem is not satisfied.
\begin{remark}\label{remark2} No other measure on $\mathcal{S}$ is preserved by $G_{\nu_n}$ because if $v$ is any non zero
 vector from $\mathbb{R}^2$ then $\lim_{j\to\infty}(g\tilde{g}^{-1})^j.v=\tilde{e}$, where
and  $\tilde{e}$ is either $e$ or $-e.$
\end{remark}

We shall now show that, in contrast, Corollary \ref{Thm3} with $k=2$ implies the exponential growth
of $S_n$ under very mild additional condition on the joint distribution of $(q_{n-1}, q_n)$.

\begin{proposition} Suppose that conditions $\mathrm{(a)}$ and $\mathrm{(b)}$ are satisfied and in addition

\smallskip\noindent
$\mathrm{(c)}$ for all $n\ge1$, the support of the joint distribution
of $(q_{n-1}, q_n)$ contains two points with the same first coordinate, say $(y,z)$  and $(y,\tilde{z})$
(which may depend on $n$).

\smallskip\noindent
Then the product $S_n$ grows exponentially.
\end{proposition}
\begin{proof} To be able to use Corollary \ref{Thm3}, we first define the set $M$.
Let $\mathcal{Z}$ be the set of all two-dimensional distributions of pairs of random variables $(z_1, z_2)$
such that $|z_i|\le C$ and $\Var(z_i)\ge\delta$, $i=1,\,2$. For each such pare $(z_1, z_2)$ denote by
$\nu_{(z_1, z_2)}$ the distribution on $\mathrm{SL}(2,\mathbb{R})$ 
of the product $g_{z_2}g_{z_1}$, where $g_{z_i}$ are matrices of the form \eqref{M1}
with $a_nq_n$ replaced by $z_i$, $i=1,2$. 
Finally,
\[
M=\{\nu_{(z_1, z_2)}:\text{the distribution of $(z_1, z_2)$ belongs to $\mathcal{Z}$}\}.
\]
In words, $M$ is the set all distributions $\nu_{(z_1, z_2)}$ described above.
Conditions imposed on $z_1,\ z_2$  imply that $M$ is a weakly compact set.
Property (b) implies that the distribution of the product of the product $g_ng_{n-1}$ belongs to $M$.
The assumption II(a) of Corollary \ref{Thm3} is thus satisfied.

Our next Lemma \ref{LemmaAssIIa} proves that also assumption II(b) is satisfied and hence
the product $S_n$ grows exponentially.
\end{proof}

\begin{lemma}\label{LemmaAssIIa} If conditions $\mathrm{(b)}$ and $\mathrm{(c)}$ are satisfied then no probability measure on
$\mathcal{S}$ is preserved by the group $G_{\nu_{n-1,1}}$.
\end{lemma}
\begin{proof}
Let $(g_y,\, g_{z})$, $(g_{y},\,g_{\tilde{z}} )$, and $ (g_{\bar{y}},g_{\bar{z}})$ be three pairs of matrices
corresponding to the points $(y,z)$, $(y,\tilde{z})$, and $(\bar{y}, \bar{z})$ respectively.
Here the first two points are chosen from the support of the distribution of $(q_{n-1}, q_n)$
as allowed by (c) and $(\bar{y}, \bar{z})$ is one more point from the same support such that $\bar{y}\not=y$
(it exists due to condition (b)).

The product $g_zg_y(g_{\tilde{z}}g_{y})^{-1}=g_{z}g_{\tilde{z}}^{-1}\in G_{\nu_{n-1,1}}$ by the definition
of $G_{\nu_{n-1,1}}$ (see \eqref{Def:Gnul}). Hence the group $G_{\nu_{n-1,1}}$ contains a non-trivial upper triangular
subgroup. This, by Remark \ref{remark2}, implies that if there is a measure on $\mathcal{S}$
preserved by $G_{\nu_{n-1,1}}$ then it should be supported by the set $\{e,-e\}\subset \mathcal{S}$.

But we also have that $g_zg_y(g_{\bar{z}}g_{\bar{y}})^{-1}\in G_{\nu_{n-1,1}}$ and a straightforward calculation gives:
\[
g_zg_y(g_{\bar{z}}g_{\bar{y}})^{-1}e=
\left( \begin{matrix}
z & -1 \\
1 & 0
\end{matrix}\right)
\left( \begin{matrix}
1 & y-\bar{y} \\
0 & 1
\end{matrix}\right)
\left( \begin{matrix}
 0& 1 \\
-1 &\bar{z}
\end{matrix}\right)e=
\left( \begin{matrix}
z(\bar{y}-y)+1 \\
\bar{y}-y
\end{matrix}\right).
\]
We thus see that the action of $G_{\nu_{n-1,1}}$ on $\mathcal{S}$ does not map $\{e,-e\}$ into itself
and therefore no measure on $\mathcal{S}$ is preserved by $G_{\nu_{n-1,1}}$.
Lemma is proved.\end{proof}

\subsection{Perturbations of random products.}\label{SecExample2}
Suppose that conditions of Theorem \ref{Thm1} are satisfied and consider a ``distortion" of the product $S_n$ of the
form
\[
\tilde{S}_n=a_ng_na_{n-1}g_{n-1}...a_1g_1,
\]
where $a_j\in \SL,\, j\ge1,$ is a non-random sequence of bounded matrices, $\|a_n\|\le C$ for some $C$ and all $n\ge1$,
and arbitrary otherwise. We claim that then there is $\tilde{\lambda}>0$ such that with probability 1
\begin{equation}\label{M2distorted}
\liminf_{n\to\infty}\frac1n \ln\|a_ng_n\ldots a_1g_1\|\ge \tilde{\lambda}.
\end{equation}
\begin{proof}[Proof of \eqref{M2distorted}]
Set $\tilde{\mathfrak{g}}_n=a_n\mathfrak{g}_n$, where $\mathfrak{g}_n$ is the function defined in section \ref{SecMatrices}.
denote $\tilde{g}_n=\tilde{\mathfrak{g}}_n(\xi_n)=a_ng_n$; our product now is $\tilde{S}_n=\tilde{g}_n\tilde{g}_{n-1}...\tilde{g}_1$.

To be able to use Theorem \ref{Thm1} with  functions $\mathfrak{g}_n$ replaced by $\tilde{\mathfrak{g}}_n$, we shall replace
$M$ by $\tilde{M}$ chosen so that to make assumption II to be satisfied for matrices $\tilde{g}_n$.
(Note that assumption I is satisfied because the underlying Markov chain is the same.)

Namely, define
\[
\tilde{M}=\bigcup_{b: b\in\SL,\|b\|\le C} (bM),
\]
where $b\in\SL$ and $bM=\{b\nu: \nu\in M\}$. Here $b\nu$ is the distribution on $\SL$ defined for a Borel set
$\Gamma\subset\SL$ by $(b\nu)(\Gamma)=\nu(b^{-1}(\Gamma))$.

Denote $\tilde{\nu}_n$ the distribution of $\tilde{g}_n$. Since $\supp(\tilde{\nu}_n)=a_n\supp(\nu_n)\in \tilde{M}$, assumption II
is satisfied  because of the following lemma.
\begin{lemma}
$\mathrm{(i)}$ $\tilde{M}$ is a compact set. \newline
$\mathrm{(ii)}$ For any $\tilde{\nu}\in \tilde{M}$ one has: no measure on $\mathcal{S}$ is preserved by $G_{\tilde{\nu}}$.
\end{lemma}
\begin{proof}
Let $\bar{\nu}_n\in\tilde{M},\ n\ge1,$ be a sequence of distributions such that $\lim_{n\to\infty}\bar{\nu}_n=\bar{\nu}$.
Then there are $b_n\in\SL$ and  $\nu_n'\in M$ such that $\bar{\nu}_n=b_n\nu_n'$.
By passing, if necessary, to a subsequence, we can assume that $\lim_{n\to\infty}b_n=a$, where the convergence is in norm, and
$\lim_{n\to\infty}\nu_n'=\nu'\in M$, where the convergence is understood in the weak sense.
Hence $\bar{\nu}=a\nu'\in \tilde{M}$ and (i) is proved.

Since $G_{\bar{\nu}}=aG_{\nu'} a^{-1}$, no measure on $\mathcal{S}$ is preserved by $G_{\bar{\nu}}$. This proves (ii).
\end{proof}

And thus \eqref{M2distorted} now follows from Theorem \ref{Thm1}. \end{proof}

\subsection{Products of stationary Markov-dependent matrices}\label{SecExample3}
The goal of this section is to explain how to derive Virtser's  (and thus also
Furstenberg's) theorem from Theorem \ref{Thm1}.

To state Virtser's theorem we start with a setup which is a simplified version of the one we saw in
sections \ref{SecDefMC} and \ref{SecMatrices}.

Let $(X,\mathcal{B},\mu)$ be a probability space and let $\xi=(\xi_n)_{n\ge1}$  be a stationary Markov chain with
the phase space $X$, transition kernel $k(x,dy)$, and invariant measure $\mu$.

Denote by $H$ the Hilbert space of real valued functions on $X$ which are square integrable with respect to $\mu$
and let $H^0\subset H$ be the space of functions with zero mean (as in \eqref{DefH^0}).
The related transition operator $K:H\mapsto H$ acts on $f\in H$ as follows: $(Kf)(x)=\int_X k(x,dy)f(y)$. The operator
$K^0:H^0\mapsto H^0$ is the restriction of $K$ to $H^0$.

We recall that we are interested in the study of the growth of the product
\[
S_n=\mathfrak{g}(\xi_n)...\mathfrak{g}(\xi_1), \text{ where } \mathfrak{g}:X\mapsto \SL.
\]
The distribution of $\mathfrak{g}(\xi_j)$ on $\SL$ is denoted $\nu$. The group $\mathbb{G}_\nu$ is as in
Theorem \ref{ThmF} and $G_\nu$ is as in \eqref{Def:Gnu} (but now there is no dependence on $j$).


\begin{theorem}[Virtser, \cite{V}]\label{ThmV} Suppose that: \newline
$\mathrm{(a)}$ $\xi$ is a stationary ergodic Markov chain .\newline
$\mathrm{(b)}$ $\|K^0\|=c$, where $c<1$.\newline
$\mathrm{(c)}$ $\int_X\ln\|\mathfrak{g}(x)\|\mu(dx)<\infty$.\newline
$\mathrm{(d)}$ No probability measure on $\mathcal{S}$ is preserved by $\mathbb{G}_\nu$.

Then there is $\lambda>0$ such that with probability 1 $\lim_{n\to\infty}n^{-1}\ln\|S_n\|=\lambda$.
\end{theorem}
\begin{proof}
It follows from the definitions of $\mathbb{G}_\nu$ and ${G}_\nu$ that $G_\nu\subset\mathbb{G}_\nu$.
If $G_\nu=\mathbb{G}_\nu$ then Theorem \ref{ThmV} is an immediate corollary of Theorem \ref{Thm1} with
$M$ in assumption II consisting of one point, $M=\{\nu\}$.

So, from now on we suppose that $G_\nu$ is a proper subgroup of $\mathbb{G}_\nu$.

Let $I_m\in \SL$ be the $m\times m$ identity matrix. Note that if $I_m\in\supp(\nu)$
then  $G_\nu=\mathbb{G}_\nu$. We shall show that if the Markov chain $\xi$ is stationary as in
Theorem \ref{ThmV} then a stronger version of Theorem \ref{Thm1} holds for this chain. Namely, the group $G_\nu$
can be replaced by $\mathbb{G}_\nu$. Virtser's theorem is again a corollary - but of this stronger version.
It should be emphasized that the stationarity of $\xi$ is crucial for the construction presented below.

\subsubsection{Definition of the Markov chain $\zeta$ and the function $\tilde{\mathfrak{g}}$}
Given the chain $\xi$, we now define a new Markov chain $\zeta$. 

The phase space of $\zeta$ is $\tilde{X}=X\cup \bar{X}$ where $\bar{X}=X\times\{1\}= \{(x,1): x\in X\}$.

To define the corresponding sigma algebra
$\tilde{\mathcal{B}}$ of subsets of $\tilde{X}$ we first define the sigma algebra $\bar{\mathcal{B}}$ of subsets of
$\bar{X}$ as the image of $\mathcal{B}$ under the natural one to one correspondence $x \leftrightarrow (x,1)$ between $X$ and $\bar{X}$.
We set $\tilde{\mathcal{B}}=\{B\cup \bar{B}: B\in\mathcal{B},\ \bar{B}\in\bar{\mathcal{B}}\}$.

The transition probabilities of the chain $\zeta$ are defined as follows. Choose a $p$, $0<p<1,$ and let $q=1-p$.
Next, define
\begin{equation}\label{M31}
\mathbb{P}(\zeta_{n+1}=(x,1)\,|\,\zeta_{n}=x)=\mathbb{P}(\zeta_{n+1}=(x,1)\,|\,\zeta_{n}=(x,1))=p,
\end{equation}
and, for $A\in\mathcal{B}$, $A\subset X\subset \tilde{X}$  define
\begin{equation}\label{M32}
\mathbb{P}(\zeta_{n+1} \in A\,|\,\zeta_{n}=x)=\mathbb{P}(\zeta_{n+1}\in A\,|\,\zeta_{n}=(x,1))=
q\mathbb{P}(\xi_{n+1} \in A\,|\,\xi_{n}=x).
\end{equation}
Define $\tilde{\mu}=(q\mu,p\mu)$ to be the initial distribution of $\zeta$: if $B\cup\bar{B}\in\tilde{B}$ then
\[
\mathbb{P}(\zeta_1\in B\cup\bar{B})=q\mu(B) +p\mu(\bar{B}).
\]
We shall see below (see Lemma \ref{LemmaInvMeas}) that
$\tilde{\mu}$ is also the invariant measure of $\zeta$.

Next, define $\tilde{\mathfrak{g}}:\tilde{X}\mapsto \SL$ as follows:
$\tilde{\mathfrak{g}}(\tilde{x})=\begin{cases} \mathfrak{g}(x) & \text{ if } \tilde{x}=x\in X,\\
I_m & \text{ if } \tilde{x}\in \bar{X}
\end{cases}$

\subsubsection{Derivation of Theorem \ref{ThmV}}

Let $\tau_1<\tau_2<...<\tau_n<...$ be the sequence of all random consecutive time moments at which
the chain $\zeta$ visits $X$. Set $\bar{\xi}_n=\zeta_{\tau_n}$. It is obvious from the definitions
\eqref{M31}, \eqref{M32} that the sequens $\bar{\xi}=(\bar{\xi}_n)_{n\ge1}$ is a Markov chain which
has the same transition probabilities and the same initial distribution as the chain $\xi$:
\begin{equation}\label{M36}
\xi\overset{d}{=}\bar{\xi}.
\end{equation}
Here and below $\overset{d}{=}$ means the equality of distributions.

Set
\[
\bar{S}_n=\bar{\mathfrak{g}}(\bar{\xi}_n)\bar{\mathfrak{g}}(\bar{\xi}_{n-1})...\bar{\mathfrak{g}}(\bar{\xi_1}),\ \
\tilde{S}_{n}=\bar{\mathfrak{g}}(\zeta_{n})...\bar{\mathfrak{g}}(\zeta_1)
\]
Equality \eqref{M36} implies that  $S_n\overset{d}{=}\bar{S}_n$. In turn $\bar{S}_n=\tilde{S}_{\tau_n}$ because the
factors forming $\tilde{S}_{\tau_n}$ are either equal to $I_m$ or coincide with one of the factors forming $\bar{S}_n$.
Hence
\[
\lim_{n\to\infty} \frac 1n\ln\|S_n\|=\lim_{n\to\infty} \frac 1n\ln\|\bar{S}_n\|=\lim_{n\to\infty}
\frac{\tau_n}{n}\frac{1}{\tau_n}\ln\|\tilde{S}_{\tau_n}\|,
\]
where the existence of all limits follows from Kingman's sub-additive ergodic theorem.
Thus, $\lambda = q^{-1}\tilde{\lambda}$ and it remains to check that $\tilde{\lambda}>0$.

\subsubsection{Theorem \ref{Thm1} in the setting of Theorem \ref{ThmV}}
The definitions of $\bar{\mathfrak{g}}$ and the formula for the invariant measure $\tilde{\mu}$ imply that the distribution
$\tilde{\nu}$ of $\bar{\mathfrak{g}}(\zeta_j)$ has the property $\supp(\tilde{\nu})=\supp(\nu)\cup\{I_m\}$.
So assumption II of Theorem \ref{Thm1} is satisfied with $M=\{\tilde{\nu}\}$ and because
$G_{\tilde{\nu}}=\mathbb{G}_{\tilde{\nu}}$.

To see that assumption I is satisfied, consider the Hilbert space $\tilde{H}$ of real-valued functions
\[
\tilde{H}=\{f=(\varphi,\psi): \|f\|^2= q\int_X\varphi(x)^2d\mu(x)+p\int_X\psi((x,1))^2d\mu(x)<\infty\},
\]
where $\varphi:X\mapsto \mathbb{R}$ and $\psi:\bar{X}\mapsto \mathbb{R}$ are the restrictions of $f$ to $X$ and $\bar{X}$
respectively.

Definitions \eqref{M31} and \eqref{M32} imply that the action of the transition operator $\tilde{K}$ of the chain $\zeta$
on $\tilde{H}$ is given by
\begin{equation}\label{M34}
\begin{aligned}
(\tilde{K}f)(x)&=q\int_Xk(x,dy)\varphi(y)+p\psi((x,1)), \text{ where } x\in X\\
(\tilde{K}f)((x,1))&=q\int_Xk(x,dy)\varphi(y)+p\psi((x,1)), \text{ where } (x,1)\in\bar{ X}.
\end{aligned}
\end{equation}

Formulae \eqref{M34} show that $\tilde{K}$ maps $\tilde{H}$ into its subspace which consists of functions $f=(\varphi,\psi)$ such that $\varphi(x)=\psi((x,1))$ for all $x\in X$.
Obviously this subspace is an invariant subspace of $\tilde{K}$. Denote this subspace $\bar{H}$ and let $\bar{H}^0\subset \bar{H}$ be the subspace
of functions with zero mean. Finally let $\tilde{K}^0$ be the restriction of $\tilde{K}$ to $\bar{H}^0$.
\begin{lemma}
$\|\tilde{K}^0\|_{\bar{H}}\le(qc+p)<1.$
\end{lemma}
\begin{proof} If $f=(\varphi,\varphi)$ then $\|f\|_{\tilde{H}}=\|\varphi\|_H$, where the notations for the norms emphasize
that $f\in \tilde{H}$ and $\varphi\in H$. If $f\in \bar{H}^0$ then $\mathbb{E}(f(\zeta))=\int_X\varphi(x)d\mu(x)=0$ and hence
$\varphi\in H^0$.  These two facts imply that for $f\in \bar{H}^0$
\[
\|\tilde{K}^0f\|_{\tilde{H}}=\|qK^0\varphi+p\varphi\|_H\le q\|K^0\| \|\varphi\|_H+p\|\varphi\|_H=(qc+p)\|f\|_{\tilde{H}}.
\]
\end{proof}
So, the chain $\zeta$ satisfies also assumption I of Theorem \ref{Thm1} with all $H_n^0$ replaced by $\bar{H}^0$
and hence $\tilde{\lambda}>0$.
\end{proof}

We finish this section with a proof of the fact which we have already stated and used above.
\begin{lemma}\label{LemmaInvMeas}
$\tilde{\mu}=(q\mu,p\mu)$ is an invariant measure of the chain $\zeta$.
\end{lemma}
\begin{proof} Formulae \eqref{M34} show that $(\tilde{K}f)(x)=(\tilde{K}f)((x,1))$. We also have that,
by the definition of the invariant measure for $\xi$, $\int_X (K\varphi)(x)d\mu(x)=\int_X \varphi(x)d\mu(x)$.
Now, the following calculation shows that
\[
\int_{\tilde{X}}(\tilde{K}f)(\tilde{x})d\tilde{\mu}(\tilde{x})=\int_X (\tilde{K}f)(x)d\mu(x)=
q\int_X\varphi(x)d\mu(x)+p\int_X\psi((x,1))d\mu(x)
\]
and hence $\tilde{\mu}$ is the invariant measure of the chain $\zeta$.
\end{proof}


\section{Proof of Theorem \ref{Thm1} for products of independent matrices}\label{SecIndep}

Suppose that matrices $g_n,\ n\ge1$ are independent. In this setting, it is natural to assume
that $X=\mathrm{SL}(m,\mathbb{R})$.
The kernels $k_n(x,dy)$ do not depend on $x$ and $\mu_n(dy)$ is the distribution of $g_n$.
Obviously, $\mu_n(dy)=k_n(dy)=\nu_n(dy)$ and $G_\nu=G_\mu$.

\begin{theorem}\label{ThmIndep} Suppose that $M$ is a compact set of probability measures on $\mathrm{SL}(m,\mathbb{R})$ and
that for any $\nu\in M$ the group $G_\nu$ does not preserve any probability measure on the unit
sphere $\mathcal{S}$. Then \eqref{M2} holds with probability 1 for some non-random $\lambda>0$.
\end{theorem}
\begin{proof} The proof of Theorem \ref{ThmIndep} will be carried out in three steps.
\smallskip

\textit{Step 1.} Note that in order to prove \eqref{M2} it suffices to show that there are constants $A>0$ and $\mathfrak{a}>0$ such that
\begin{equation}\label{M3}
\mathbb{E}(\|S_n\|^{-\frac m2})\le A e^{-\mathfrak{a}\, n}.
\end{equation}
Indeed, by the Markov inequality for any $\varepsilon>0$
\[
\mathbb{P}(\|S_n\|\le e^{\varepsilon n})=\mathbb{P}(\|S_n\|^{-\frac m2}\ge e^{-\frac m2\varepsilon n})
\le e^{\varepsilon\frac m2 n}\mathbb{E}(\|S_n\|^{-\frac m2})\le Ae^{(\varepsilon\frac m2-\mathfrak{a}) n}.
\]
If $\varepsilon<2\mathfrak{a}/m$ then the Borel-Cantelli lemma implies that the set
$\{n:\|S_n\|\le e^{\varepsilon n} \}$ is a.s. finite.  This means that for any $\lambda<2\mathfrak{a}/m$ \eqref{M2} holds
with probability 1.
 \smallskip

\textit{Step 2.}
Let $L_2(\mathcal{S})$ be the Hilbert space of real valued functions on $\mathcal{S}$ equipped with
the Lebesgue measure $du$ which is normalized to 1. The inner product of $f,\, h\in L_2(\mathcal{S})$ is given by
\[
\left<f,h\right>_{L_2}=\int_\mathcal{S}f(u)h(u)du.
\]
Let $\mathbb{V}$  be the set of unitary operators in $L_2(\mathcal{S})$. Consider a mapping
$V:g\mapsto V_g$, where $g\in\mathrm{SL}(m,\mathbb{R}),\ V_g\in \mathbb{V}$ and $V_g$ is defined for $f\in L_2(\mathcal{S})$
as follows:
\begin{equation}\label{M8}
(V_gf)(u)=f(g.u)||gu||^{-\frac{m}{2}}.
\end{equation}
The mapping $V$ has the following properties:
\begin{equation}\label{M7}
||f||_{L_2}=||V_gf||_{L_2}\ \ \hbox{ and }\ \ V_{g_1g_2}=V_{g_{2}}V_{g_1}.
\end{equation}
The firs relation in \eqref{M7} follows from the fact that $||gu||^{-m}$ is the Jacobian of the transformation
$u\mapsto g.u$  (see  \cite[page 425, Lemma 8.8]{F}). Therefore
$\int_{\mathcal{S}}f(g.u)^2||gu||^{-m}du=\int_{\mathcal{S}}f(u)^2du$.

The second relation in \eqref{M7} is verified by a straightforward calculation.

\begin{remark} The mapping $g\mapsto V_{g^{-1}}$ is a representation of $\mathrm{SL}(m,\mathbb{R})$
which was used in  \cite{V}. We prefer to work with  $V$ because it simplifies some formulae.

\end{remark}
For a probability measure $\nu$, put
\begin{equation}\label{M10}
W_\nu=\int_{\mathrm{SL}(m,\mathbb{R})} V_g\nu(dg), \text{ that is }\
(W_\nu f)(u)=\int_{\mathrm{SL}(m,\mathbb{R})} f(g.u)||gu||^{-\frac{m}{2}}\nu(dg).
\end{equation}

\begin{lemma}\label{Thm2} Let $M$ be a weakly compact set of probability measures on
$\textrm{SL}(m,\mathbb{R})$
such that each $\nu\in M$ has the property that no probability
measure on $\mathcal{S}$ is preserved by $G_\nu$. Then there is a constant $\beta,\ 0\le \beta <1,$
such that $\|W_\nu\|\le\beta$ for all $\nu\in M$.
\end{lemma}
We shall prove Lemma \ref{Thm2} after we finish the proof of our Theorem.

\textit{Step 3.} Since $\|S_n\|\ge \|S_nu\|$, $u\in\mathcal{S}$, \eqref{M3} would follow from
\begin{equation}\label{M4}
\int_\mathcal{S}\mathbb{E}(\|S_nu\|^{-\frac m2})du\le e^{-c n}.
\end{equation}
Note next that
\begin{equation}\label{M9}
\|g_n...g_1u\|^{-\frac m2}= (V_{g_n...g_1}\mathbf{1})(u)=(V_{g_1}...V_{g_n}\mathbf{1})(u),
\end{equation}
where $\mathbf{1}$ is the function on $\mathcal{S}$ which takes value 1 at every $u\in\mathcal{S}$.
Therefore
\begin{equation}\label{M5}
\begin{aligned}
&\int_\mathcal{S}\mathbb{E}(\|S_nu\|^{-\frac m2})du=\mathbb{E}\left(\int_\mathcal{S}\|g_n...g_1u\|^{-\frac m2}du\right)\\
&=\mathbb{E}\left(\int_\mathcal{S}(V_{g_1}...V_{g_n}\mathbf{1})(u)du\right)
=\mathbb{E}\left(\left<V_{g_1}...V_{g_n}\mathbf{1},\mathbf{1}\right>\right)=\left<\mathbb{E}\left(V_{g_1}...V_{g_n}\right)\mathbf{1},\mathbf{1}\right>.
\end{aligned}
\end{equation}
Since the operators $V_{g_1},...,V_{g_n}$ are independent we obtain
\[
\mathbb{E}\left(V_{g_1}...V_{g_n}\right)=\mathbb{E}\left(V_{g_1}\right)...\mathbb{E}\left(V_{g_n}\right)
=W_{\nu_1}...W_{\nu_n}
\]
Finally,
\begin{equation}\label{M6}
\int_\mathcal{S}\mathbb{E}(\|S_nu\|^{-\frac m2})du=\left<W_{\nu_1}...W_{\nu_n}\mathbf{1},\mathbf{1}\right>\le
\|W_{\nu_1}\|...\|W_{\nu_n}\|\le e^{-c n},
\end{equation}
where $c=\inf_{\nu\in M}(- \ln\|W_\nu\|)>0$ by Lemma \ref{Thm2}.

This completes the proof of Theorem \ref{ThmIndep}.
\end{proof}
\subsection{Proof of Lemma \ref{Thm2}}
\begin{proof} We shall prove the following statement which is equivalent to Lemma \ref{Thm2}.
\begin{equation}\label{LemmaMain}
\begin{aligned}
\text{If $\sup_{\nu\in M}\|W_\nu\|=1$} &\text{ then there is a $\nu\in M$ and a probability measure } \kappa \\
& \text{ on $\mathcal{S}$ which is preserved by $G_\nu$.}
\end{aligned}
\end{equation}
From $\sup_{\nu\in M}\|W_\nu\|=1$ it follows that there is a sequence of measures $\nu_n\in M$ and a sequence of functions
$f_n\in L_2(\mathcal{S})$ with $\|f_n\|=1$ and such that $\lim_{n\to\infty}||W_{\nu_n} f_n||=1$.
Define a sequence of probability measures $\kappa_n$ on $\mathcal{S}$ by setting $\kappa_n(du)=f_n(u)^2du$.

We shall assume that both sequences of measures, $\nu_n$ and $\kappa_n$, have weak limits:
\begin{equation}\label{seqs}
\lim_{n\to\infty}\ \nu_n=\nu, \text{ where } \nu\in M, \text{ and } \lim_{n\to\infty}\kappa_n =\kappa.
\end{equation}
As usual, if \eqref{seqs} is not satisfied then the sequence $(\nu_n, \kappa_n)$ can be replaced by its subsequence
which has these properties. For the sequence $\nu_n$, this is possible because of the condition that $M$ is a compact set.
For the sequence $\kappa_n$, the existence of a converging subsequence follows from the fact that these
measures are defined on the unit sphere $\mathcal{S}$ which is a compact metric space.

We shall prove that $\kappa$ and $\nu$ are such that
\begin{equation}\label{M12}
g_1^{-1}\kappa= g_2^{-1}\kappa\ \text{ for any }\ g_1,\, g_2\in \supp(\nu).
\end{equation}
Thus $g_1g_2^{-1}\kappa=\kappa$ and hence $G_\nu$ preserves $\kappa$.

Let us rewrite \eqref{M12} in terms of test functions: we have to prove that
\begin{equation}\label{M21}
\int_{\mathcal{S}}\psi(u)(g_1^{-1}\kappa)(du)=\int_{\mathcal{S}}\psi(u)(g_2^{-1}\kappa)(du)
\end{equation}
for any $g_1,\, g_2\in \supp(\nu)$ and any continuous function $\psi:\mathcal{S}\mapsto\mathbb{R}$.
The last equality can be rewritten in another equivalent form (see Appendix, section \ref{proofM13}.):
\begin{equation}\label{M13}
\int_{\mathcal{S}}\psi(g_1^{-1}.u)\kappa(du)=\int_{\mathcal{S}}\psi(g_2^{-1}.u)\kappa(du).
\end{equation}
Let $B(g_i,\delta)=\{g\in \SL): \|g_i-g\|< \delta\}$
be the open balls and let $S(g_i,\delta)=\{g\in \SL): \|g_i-g\|= \delta\}$ be the sphere of radius $\delta$
centered at $g_i,\ i=1,2$.

It follows from the continuity of $\psi$ and the compactness of $\mathcal{S}$ that
for a given $\varepsilon>0$ there is a $\delta>0,\ \delta=\delta(\varepsilon,g_1,g_2, \psi),$ such that for all
$g_1'\in B(g_1,\delta)$ and for all $g_2'\in B(g_2,\delta)$
\begin{equation}\label{M14}
\sup_{u\in\mathcal{S}}|\psi(g_1^{-1}.u)-\psi(g_1'^{-1}.u)|\le \varepsilon\ \text{ and }
\sup_{u\in\mathcal{S}}|\psi(g_2^{-1}.u)-\psi(g_2'^{-1}.u)|\le \varepsilon.
\end{equation}
In addition, we shall suppose that $\delta$ is such that
\begin{equation}\label{M25}
\nu(S(g_i,\delta))=0
\end{equation}
Note that $\nu(B(g_1,\delta))>0$ and $\nu(B(g_2,\delta))>0$ because $g_1,g_2\in \supp(\nu)$.
The weak convergence of $\nu_n$ to $\nu$ together with \eqref{M25} imply that
\begin{equation}\label{M15}
\lim_{n\to\infty}\nu_n(B(g_i,\delta))= \nu(B(g_i,\delta))
\text{ for } i=1,\,2.
\end{equation}
Next, we shall show that  for a given $\epsilon>0$ there is $N(\epsilon)$
such that for all $n\ge N(\epsilon)$ there are
$\tilde{g}_1\in B(g_1,\delta)$ and $\tilde{g}_2\in B(g_2,\delta)$ such that the following inequalities hold:

\begin{equation}\label{M16}
I_i=\left|\int_{\mathcal{S}}\psi(g_i^{-1}.u)\kappa(du)-\int_{\mathcal{S}}\psi(\tilde{g}_i^{-1}.u)\kappa_n(du)\right|<\epsilon,\quad i=1,\,2,
\end{equation}
\begin{equation}\label{M18}
I_3=\left|\int_{\mathcal{S}}\psi(\tilde{g}_1^{-1}.u)\kappa_n(du)-\int_{\mathcal{S}}\psi(\tilde{g}_2^{-1}.u)\kappa_n(du)\right|<\epsilon. 
\end{equation}
It follows from \eqref{M16} and \eqref{M18} that
\[
\left|\int_{\mathcal{S}}\psi(g_1^{-1}.u)\kappa(du)-\int_{\mathcal{S}}\psi(g_2^{-1}.u)\kappa(du)\right|\le
I_1+I_3+I_2<3\epsilon
\]
and, since $\epsilon$ can be arbitrarily small, the last inequality proves \eqref{M13}.

It thus remains to prove \eqref{M16} and \eqref{M18}. To prove \eqref{M16} we write
\begin{equation}\label{M20}
\begin{aligned}
&\left|\int_{\mathcal{S}}\psi(g_i^{-1}.u)\kappa(du)-\int_{\mathcal{S}}\psi(\tilde{g}_i^{-1}.u)\kappa_n(du)\right|\le\\
& \left|\int_{\mathcal{S}}\psi(g_i^{-1}.u)\kappa(du)-\int_{\mathcal{S}}\psi(g_i^{-1}.u)\kappa_n(du)\right|+
\left|\int_{\mathcal{S}}(\psi(g_i^{-1}.u)-\psi(\tilde{g}_i^{-1}.u))\kappa_n(du)\right|\le\\
&\left|\int_{\mathcal{S}}\psi(g_i^{-1}.u)\kappa(du)-\int_{\mathcal{S}}\psi(g_i^{-1}.u)\kappa_n(du)\right|+\varepsilon,
\end{aligned}
\end{equation}
where the last inequality is due to \eqref{M14}. This, together with the weak convergence of $\kappa_n$ to $\kappa$,
implies \eqref{M16}.

Remark that \eqref{M16} holds for  all $\tilde{g}_1$ and $\tilde{g}_2$ from $B(g_1,\delta)$ and $B(g_2,\delta)$ respectively.

We now turn to \eqref{M18}. Note first that for any $g\in \mathrm{SL}(m,\mathbb{R})$
\begin{equation}\label{4}
\int_{\mathcal{S}}\psi(g^{-1}.u)\kappa_n(du)=\int_{\mathcal{S}}\psi(g^{-1}.u)f_n(u)^2du=
\int_{\mathcal{S}}\psi(u)f_n(g.u)^2\|gu\|^{-m}du.
\end{equation}
The last equality in \eqref{4} follows from the change of variables
$u\mapsto g.u$ since the corresponding Jacobian is $\|gu\|^{-m}$.
(It can also be viewed as one more version of the definition of $g\kappa_n$ in the case
when $\kappa_n$ has a density function $|f_n|^2$.)

Using \eqref{4}, we present the left hand side of \eqref{M18} as
\begin{equation}
I_3= \left|\int_{\mathcal{S}}\psi(u)(f_n(\tilde{g}_1.u)^2\|\tilde{g}_1u\|^{-m}-f_n(\tilde{g}_2.u)^2\|\tilde{g}_2u\|^{-m})du\right|.
\end{equation}
Denote $f_n(g)= V_gf_n$ and define
\[
\varphi_n= W_{\nu_n} f_n=\int_{\mathrm{SL}(m,\mathbb{R})}f_n(g)\nu_n(dg).
\]
Since $\|f_n(g)\|_{L_2}=1$ and $\|\varphi_n\|_{L_2} \to 1$,
the uniform convexity of the unit sphere in $L_2(\mathcal{S})$ implies that for any $\epsilon>0$
\[
\lim_{n\to\infty}\nu_n\{g: \|\varphi_n-f_n(g)\|_{L_2}>\epsilon\}=0.
\]
We thus can choose $N_1=N_1(\epsilon,\delta,g_1,g_2)$ such that for $i=1,2$ and all $n\ge N_1$
\[
\nu_n(\{g: \|\varphi_n-f_n(g)\|_{L_2}\le\epsilon\}\cap B(g_i,\delta))>0.5\nu(B(g_i,\delta)).
\]
Hence, for every $n\ge N_1$ there are $\tilde{g}_1\in B(g,\delta)$ and $\tilde{g}_2\in B(g_2,\delta)$ such that
$\|f_n(\tilde{g}_1)-f_n(\tilde{g}_2)\|_{L_2} \le \epsilon$. But then, for these $\tilde{g}_1,\  \tilde{g}_2$ (which may depend on $n$),
we have
\begin{equation}\label{M19}
\begin{aligned}
&\left|\int_{\mathcal{S}}\psi(u)(f_n(\tilde{g}_1.u)^2\|\tilde{g}_1u\|^{-m}-f_n(\tilde{g}_2.u)^2\|\tilde{g}_2u\|^{-m})du\right|\\
&\le \sup_{u\in \mathcal{S}}|\psi(u)|\int_{\mathcal{S}}\left|f_n(\tilde{g}_1.u)^2\|\tilde{g}_1u\|^{-m}-f_n(\tilde{g}_2.u)^2\|\tilde{g}_2u\|^{-m}\right|du\\
&=\sup_{u\in \mathcal{S}}|\psi(u)|\int_{\mathcal{S}}\left|f_n(\tilde{g}_1)^2-f_n(\tilde{g}_2)^2\right|du\\
&=
\sup_{u\in \mathcal{S}}|\psi(u)|\left<\left|f_n(\tilde{g}_1)-f_n(\tilde{g}_2)\right|,\left|f_n(\tilde{g}_1)+f_n(\tilde{g}_2)\right|\right>_{L_2}
\end{aligned}
\end{equation}
Since $\left\|\;|f_n(\tilde{g}_1)+f_n(\tilde{g}_2)|\;\right\|_{L_2}\le 2$ we obtain
\[
\left<\left| f_n(\tilde{g}_1)-f_n(\tilde{g}_2)\right|,\left|f_n(\tilde{g}_1)+f_n(\tilde{g}_2)\right|\right>_{L_2}
\le 2\left\|f_n(\tilde{g}_1)-f_n(\tilde{g}_2)\right\|_{L_2}\le 2\epsilon.
\]
This proves \eqref{M18} and completes the proof of the Lemma.
\end{proof}

\smallskip\noindent
\subsection{Comments}

1. In the context of products of matrices, operators $W_\nu$
were first explicitly defined in \cite{V} where it was proved that the spectral radius of $W_\nu$ is less than 1.
In the case of identically distributed independent $g_n$ this fact implies Theorem \ref{ThmIndep}.
In fact, \cite{V} starts with a more complicated version of this operator which allows one to control
products of stationary Markov-dependent matrices and, once again, the positivity of the Lyapunov
exponent follows from the fact that the corresponding spectral radius is less than 1.

2. We are now in a position to state in a more precise way the result from \cite{W1} mentioned in the Introduction.
Namely, for the special case of independent matrices $g_j$ given by \eqref{M1}, it was proved there that
$\|W_\nu W_{\tilde{\nu}}\|\le a<1$, where $a$ is explicitly expressed  in terms of the variances of $q_n$'s.

Corollary \ref{Thm3} guarantees exponential growth for a more general class of potentials because the sequence $q_n$
is only required to be Markov-dependent. On the other hand, obtaining a constructive estimate
for $||W_\nu W_{\tilde{\nu}}||$ similar to the one in \cite{W1} requires additional work.

3. Furstenberg's theorem for the i.i.d. case can be derived directly from Theorem \ref{Thm2}.
This derivation is much more straightforward than the one for the Markov-dependent matrices
discussed in section \ref{SecExample3}.
Namely, we have again to consider two cases. If the identity matrix $I_m\in\supp(\nu)$ then
$G_\nu=\mathbb{G}_\nu$ and the positivity claimed by Theorem \ref{ThmF} follows from Theorem \ref{Thm2}.
If $I_m\not\in\supp(\nu)$, then we can apply our theorem to the measure
$\tilde{\nu}=\frac12\nu+\frac12\delta_I$. Obviously,
$I_m\in\supp(\tilde{\nu})$ and by Theorem \ref{Thm2} the corresponding $\tilde{\lambda}>0$ .
Since $\lambda=2\tilde{\lambda}$ (as in section \ref{SecExample3})
the result follows. This completes the proof of the main claim of Furstenberg's theorem.

\section{Proof of Theorem \ref{Thm1}}\label{SecMarkov}
The plan of the proof is as follows.

As in the case of Theorem \ref{ThmIndep}, our aim is to show that inequality \eqref{M4} holds.
To this end, we first introduce a sequence of Hilbert spaces $\mathbb{H}_{n}$ which are extensions of the
spaces $H_n$ (defined by \eqref{2.H}) and operators  $\hat{K}_j$ and $\hat{V}_{\mathfrak{g}_j}$
which are the analogues of $K_j$ and $V_{g_j}$.
We then compute $\int_{\mathcal{S}}\mathbb{E}\left(\|S_nu\|^{-1}\right)du$ in terms of products of these operators,
state our main technical estimate (Theorem \ref{ThmTechn}) and use it to prove Theorem \ref{Thm1}.
The proof of Theorem \ref{ThmTecn} is given after that, in section \ref{SecMainTechn}.
As in the case of independent matrices, Lemma \ref{Thm2} plays an important role in the proof.

\subsection{Auxiliary spaces and operators and the proof of Theorem \ref{Thm1}}
Denote by $\mathbb{H}_n$ the Hilbert space of $\mu_n\times du$-square integrable real valued functions on
$X\times\mathcal{S}$:
with the standard inner product: if $f,\, h\in \mathbb{H}_n$ then $\left<f,h\right>_{\mathbb{H}_n}=\int_{X\times\mathcal{S}}f(x,u)h(x,u)\mu_n(dx)du.$

The spaces $H_n$ and $L_2(\mathcal{S})$ are naturally imbedded into $\mathbb{H}_n$.
The image of the natural imbedding of $L_2(\mathcal{S})$ into $\mathbb{H}_n$ will be denoted by $\mathcal{L}_n$.
Obviously, $\mathcal{L}_n$ consists of functions from $\mathbb{H}_n$ which depend only on $u\in \mathcal{S}$.

\subsubsection{Operators $\hat{K}_n$ and $\hat{V}_\mathfrak{g_n}$}\label{Sec4.1.1} We first extend the action of $K_n$ to
$\mathbb{H}_{n+1}$.  Namely, denote by $\hat{K}_n:\mathbb{H}_{n+1}\mapsto \mathbb{H}_n$ the operator
which, for $f\in\mathbb{H}_{n+1}$, is defined by
\begin{equation}\label{3.K}
(\hat{K}_nf)(x,u)=\int_{X}k_n(x,dy)f(y,u).
\end{equation}
\begin{remark} \label{remark}
If $f\in \mathcal{L}_{n+1}$ (that is $f(x,u)\equiv f(u)$) then,
by the definition of $\hat{K}_n$, $(\hat{K}_nf)(x,u)=f(u)\in \mathcal{L}_{n}$.
\end{remark}
Next, we define unitary operators $\hat{V}_\mathfrak{g_n}:\mathbb{H}_{n}\mapsto \mathbb{H}_{n}$.
Namely, for $f\in\mathbb{H}_{n}$ we set
\begin{equation}\label{3.V}
(\hat{V}_{\mathfrak{g}_n}f)(x,u)=f(x,\mathfrak{g}_n(x).u)\|\mathfrak{g}_n(x)u\|^{-\frac m2},
\end{equation}
where $\mathfrak{g}_n$ are the functions on $X$ introduced at the beginning of subsection \ref{SecMatrices}.
\begin{remark}
The operators $V_g$ defined by \eqref{M8} act on $\mathbb{H}_n$ in a natural way. Namely, $(V_gf)(x,u)=
f(x,g.u)\|gu\|^{-\frac m2}$. Clearly, $\hat{V}_{\mathfrak{g}_n}=V_g$ if and only if $\mathfrak{g}_n(x)\equiv g$.
\end{remark}
The following lemma and especially its corollary are a variation and an extension of Lemma 3 from \cite{V}
to the case of products of non-stationary Markov-dependent matrices. 
\begin{lemma} Consider the Markov chain $(\xi_j)_{j\ge1}$ and let $f\in \mathbb{H}_n$, $u\in\mathcal{S}$. Then for $n\ge 2$
\begin{equation}\label{product}
\begin{aligned}
&\mathbb{E}\{\|\mathfrak{g}_n(\xi_n)...\mathfrak{g}_2(\xi_2)u\|^{-\frac m2}f(\xi_n,\mathfrak{g}_n(\xi_n)...
\mathfrak{g}_2(\xi_2).u)\big|\xi_1=x\}=\\
&(\hat{K}_1\hat{V}_{\mathfrak{g}_2}...\hat{K}_{n-1}\hat{V}_{\mathfrak{g}_n}f)(x,u).
\end{aligned}
\end{equation}
\end{lemma}
\begin{proof}
By the above definitions
\[
(\hat{K}_{n-1}\hat{V}_{\mathfrak{g}_n}f)(x,u)=\int_{X}k_{n-1}(x,dy)f(y,\mathfrak{g}_n(y).u)\|
\mathfrak{g}_n(y)u\|^{-\frac m2}.
\]
A straightforward induction argument now implies that
\begin{equation}\label{product1}
\begin{aligned}
&(\hat{K}_1\hat{V}_{\mathfrak{g}_2}...\hat{K}_{n-1}\hat{V}_{\mathfrak{g}_n}f)(x,u)=\\
&\int_{X^{n-1}}k_1(x,dy_2)...k_{n-1}(y_{n-1}, dy_n)
f(y_n,\mathfrak{g}_n(y_n)...\mathfrak{g}_2(y_2).u)
\|\mathfrak{g}_n(y_n)...\mathfrak{g}_2(y_2)u\|^{-\frac m2},
\end{aligned}
\end{equation}
where $X^{n-1}= X\times...\times X$ is the $(n-1)^{\mathrm{th}}$ direct power of $X$.
The right hand side of the last formula coincides with the definition of the expectation in the left hand side
of \eqref{product}. This proves the Lemma.
\end{proof}
Recall that $S_n=\mathfrak{g}_n(\xi_n)...\mathfrak{g}_2(\xi_2)\mathfrak{g}_1(\xi_1).$
Applying $\hat{V}_{\mathfrak{g}_1}$ to both parts of \eqref{product} we obtain:
\begin{equation}\label{product2}
\mathbb{E}\{\|S_nu\|^{-\frac m2}f(\xi_n,S_n.u)\big|\xi_1=x\}=
(\hat{V}_{\mathfrak{g}_1}\hat{K}_1\hat{V}_{\mathfrak{g}_2}...\hat{K}_{n-1}\hat{V}_{\mathfrak{g}_n}f)(x,u).
\end{equation}
In turn, \eqref{product2} implies the following analogue of \eqref{M5}:
\begin{corollary}
\begin{equation}\label{product3}
\int_\mathcal{S}\mathbb{E}(\|S_nu\|^{-\frac m2})du=\left<\hat{V}_{\mathfrak{g}_1}\hat{K}_1\hat{V}_{\mathfrak{g}_2}...
\hat{K}_{n-1}\hat{V}_{\mathfrak{g}_n}\mathbf{1}_n,\mathbf{1}_1\right>_{\mathbb{H}_1},
\end{equation}
where $\mathbf{1}_n\in\mathbb{H}_n$ and $\mathbf{1}_1\in\mathbb{H}_1$ are functions taking the value 1
at all points of their respective domains.
\end{corollary}
\begin{proof}
Replace $f$ in \eqref{product2} by $\mathbf{1}_n$ and integrate both sides of \eqref{product2} over $\mu_1(dx)\times du$.
\end{proof}
The following theorem is the main technical result of this paper.
\begin{theorem}\label{ThmTechn} Suppose that assumptions I and II are satisfied. Then there is a positive
constant $\alpha<1$ such that for all $n$
\begin{equation}\label{MainTechn2}
\|\hat{K}_{n-1}\hat{V}_{\mathfrak{g}_n}\hat{K}_{n}\hat{V}_{\mathfrak{g}_{n+1}}\|\le\alpha.
\end{equation}
\end{theorem}
Theorem \ref{ThmTechn} will be proved in the next section. We finish this section with the
\begin{proof}[Proof of Theorem \ref{Thm1}] Relations \eqref{product3} and \eqref{MainTechn2}
imply for odd $n\ge3$ that
\begin{equation}\label{product4}
\begin{aligned}
\int_\mathcal{S}\mathbb{E}(\|S_nu\|^{-\frac m2})du &\le
\|\hat{V}_{\mathfrak{g}_1}\hat{K}_1\hat{V}_{\mathfrak{g}_2}...\hat{K}_{n-1}\hat{V}_{\mathfrak{g}_n}\|\le
\|\hat{K}_1\hat{V}_{\mathfrak{g}_2}...\hat{K}_{n-1}\hat{V}_{\mathfrak{g}_n}\|\\
&\le \|\hat{K}_1\hat{V}_{\mathfrak{g}_2}\hat{K}_2\hat{V}_{\mathfrak{g}_3}\|...
\|\hat{K}_{n-2}\hat{V}_{\mathfrak{g}_{n-1}}\hat{K}_{n-1}\hat{V}_{\mathfrak{g}_n}\|\le \alpha^{\frac {n-1}{2}}.
\end{aligned}
\end{equation}
If $n\ge5$ is even then, similarly,
\[
\int_\mathcal{S}\mathbb{E}(\|S_nu\|^{-\frac m2})du \le
\|\hat{K}_3\hat{V}_{\mathfrak{g}_4}\hat{K}_4\hat{V}_{\mathfrak{g}_5}\|...
\|\hat{K}_{n-2}\hat{V}_{\mathfrak{g}_{n-1}}\hat{K}_{n-1}\hat{V}_{\mathfrak{g}_n}\|\le \alpha^{\frac{n-2}{2}}.
\]
We thus see that for all $n$, with the obvious choice of $A>0$ and $\mathfrak{a}>0$,
\[
\mathbb{E}(\|S_n\|^{-\frac m2}) \le \int_\mathcal{S}\mathbb{E}(\|S_nu\|^{-\frac m2})du \le  A\, e^{-\mathfrak{a}\,n}
\]
and hence exactly the same argument as in the Step 1 of the proof of Theorem \ref{ThmIndep} in section \ref{SecIndep}
finishes the proof.
\end{proof}
\subsection{Proof of Theorem \ref{ThmTechn}}\label{SecMainTechn} Remark that since $\hat{V}_{\mathfrak{g}_{n+1}}$ is
a unitary operator, the inequality \eqref{MainTechn2} is equivalent to
\begin{equation}\label{MainTechn2a}
\|\hat{K}_{n-1}\hat{V}_{\mathfrak{g}_n}\hat{K}_{n}\|\le\alpha.
\end{equation}
We shall prove this inequality  in section \ref{Sec4.2.2}. But first we prove some preparatory results.

\subsubsection{Properties of $\hat{K}_n$ and related operators}
Let
$\mathbb{H}_n^0=\{f\in \mathbb{H}_n:\;  \left<f,h\right>_{\mathbb{H}_n}=0\ \forall\, h\in \mathcal{L}_{n} \}$
be the orthogonal complement of $\mathcal{L}_{n}$ in $\mathbb{H}_n$.

Denote by $\mathcal{P}_n$ the orthogonal projector on $\mathbb{H}_n^0$ and by
$\mathcal{Q}_{n}=\mathcal{I}-\mathcal{P}_{n}$ the orthogonal projector on the subspace $\mathcal{L}_n$
(here $\mathcal{I}$ is the identity operator in $\mathbb{H}_{n}$).
We remark that if $f\in \mathbb{H}_n$  then
\begin{equation}\label{Def:projection}
(\mathcal{P}_nf)(x,u)=f(x,u)-\int_Xf(y,u)\mu_n(dy).
\end{equation}
For the proof of \eqref{Def:projection} see Appendix, section \ref{sec5.1}. Obviously,
\begin{equation}\label{Def:projection1}
(\mathcal{Q}_nf)(x,u)=\int_Xf(y,u)\mu_n(dy).
\end{equation}
Set $\hat{K}_n^0=\hat{K}_n\mathcal{P}_{n+1}$ and $\hat{K}_n^1=\hat{K}_n\mathcal{Q}_{n+1}$.
The following lemma lists several simple but useful properties of $\hat{K}_n$, $\hat{K}_n^0,\ \hat{K}_n^1$.
We remark that property (iii) is a version of Lemma 2 from \cite{V}.
\begin{lemma}\label{Lemma<c}

$\mathrm{(i)}$ $\hat{K}_n^0(\mathbb{H}_{n+1})\subset \mathbb{H}_{n}^0$.

$\mathrm{(ii)}$ $\hat{K}_n^1(\mathbb{H}_{n+1})= \mathcal{L}_{n}$.

$\mathrm{(iii)}$ If $\|K_n^0\|\le c$ then also $\|\hat{K}_n^0\|\le c$.

$\mathrm{(iv)}$ Suppose that $\|K_n^0\|\le c<1$ and let $f\in \mathbb{H}_{n+1}$ be such that
$\|\hat{K}_nf\|^2\ge\|f\|^2(1-\epsilon)$, where $0\le\epsilon\le 1$. Then
\begin{equation}\label{4.iv}
\|\mathcal{P}_{n+1}f\|^2\le \frac{\epsilon}{1-c^2}\|f\|^2\text{ or, equivalently, }
\|\mathcal{Q}_{n+1}f\|^2\ge \left(1-\frac{\epsilon}{1-c^2}\right)\|f\|^2.
\end{equation}
\end{lemma}
\begin{remark}
The equivalence in \eqref{4.iv} is due to $\|\mathcal{P}_{n+1}f\|^2+\|\mathcal{Q}_{n+1}f\|^2=\|f\|^2$.
We state two inequalities because both of them will be referred to below.
\end{remark}
\begin{proof}[Proof of $\mathrm{(i)}$] Consider an $f\in \mathbb{H}_{n+1}$. Set $\phi(y,u)=f(y,u)-\int_Xf(z,u)\mu_{n+1}(dz)$.
By the definition of $\hat{K}_n^0$
\[
(\hat{K}_n^0f)(x,u)=(\hat{K}_n\mathcal{P}_{n+1}f)(x,u)=
\int_X\phi(y,u)k_n(x,dy).
\]
If $h\in \mathcal{L}_{n+1}$ then
\[
\left<\hat{K}_n^0f,h\right>_{\mathbb{H}_n}=\int_{X\times \mathcal{S}}\left(\int_X\phi(y,u)k_n(x,dy)\right)h(u)\mu_n(dx)du.
\]
Changing the order of integration in the last formula, we obtain:
\[
\begin{aligned}
\left<\hat{K}_n^0f),h\right>_{\mathbb{H}_n}=&\int_{\mathcal{S}}\left(\int_X\left(\int_X\phi(y,u)k_n(x,dy)\right)\mu_n(dx)\right)h(u)du\\
=&\int_{\mathcal{S}}\left(\int_X\phi(y,u)\mu_{n+1}(dy)\right)h(u)du=0,
\end{aligned}
\]
where the equality $\int_X\left(\int_X\phi(y,u)k_n(x,dy)\right)\mu_n(dx)=\int_X\phi(y,u)\mu_{n+1}(dy)$ is an equivalent version
of \eqref{1.H} (with $n-1$ replaced by $n$). We also use that $\int_X\phi(y,u)\mu_{n+1}(dy)=0$.
\end{proof}

\begin{proof}[Proof of $\mathrm{(ii)}$]
$\mathcal{Q}_{n+1}f\in \mathcal{L}_{n+1}$ for any $f\in \mathbb{H}_{n+1}$ by the definition $\mathcal{Q}_{n+1}$.
It follows then from \eqref{3.K} that $\hat{K}_n^1f=\hat{K}_n\mathcal{Q}_{n+1}f \in \mathcal{L}_{n}$.
\end{proof}

\begin{proof}[Proof of $\mathrm{(iii)}$] Since $\mathcal{P}_{n+1}f\in \mathbb{H}_{n+1}^0$
it suffices to prove that $\|\hat{K}_nf\|\le c \|f\|$ for $f\in \mathbb{H}_{n+1}^0$. So for the rest of this proof
we assume that $f\in \mathbb{H}_{n+1}^0$.

For such functions $\int_Xf(y,u)\mu_{n+1}(dy)=0$ for Lebesgue-a.e. $u\in\mathcal{S}$ (see Appendix,
Lemma \ref{A1}) which means that $f(\cdot,u)\in H_{n+1}^{(0)}$ for each $u$.
Since
$\left(\hat{K}_n^0f\right)(x,u)=\int_Xk_n(x,dy)f(y,u)\in H_n^{(0)}$
the condition of our Lemma implies that for these $u$
\[
\int_X\left[\left(K_nf\right)(x,u)\right]^2\mu_n(dx)\le c^2 \int_Xf(x,u)^2\mu_{n+1}(dx).
\]
Integrating both parts of this inequality over $u\in\mathcal{S}$ we obtain
\[
\|\hat{K}_n^0f\|_{\mathbb{H}_n}^2\le c^2 \int_{X\times\mathcal{S}}f(x,u)^2\mu_{n+1}(dx)du=c^2\|f\|_{\mathbb{H}_{n+1}}^2
\]
which finishes the proof of part (iii).\end{proof}
\begin{proof}[Proof of $\mathrm{(iv)}$] $f=\mathcal{P}_{n+1}f+\mathcal{Q}_{n+1}f$ for any $f\in \mathbb{H}_{n+1}$ and hence $\hat{K}_{n}f=\hat{K}_{n}\mathcal{P}_{n+1}f+\hat{K}_{n}\mathcal{Q}_{n+1}f$. By
properties (i) and (ii), the function $\hat{K}_{n}\mathcal{P}_{n+1}f=\hat{K}_{n}^0f\in\mathbb{H}_n^0$ is orthogonal to
$\hat{K}_{n}\mathcal{Q}_{n+1}f=\hat{K}_{n}^1f\in \mathcal{L}_{n} $. Therefore
\[
\|\hat{K}_{n}f\|^2=\|\hat{K}_{n}\mathcal{P}_{n+1}f\|^2+\|\hat{K}_{n}\mathcal{Q}_{n+1}f\|^2
=\|\hat{K}_{n}^0f\|^2+ \|\mathcal{Q}_{n+1}f\|^2,
\]
where the equality $\|\hat{K}_{n}\mathcal{Q}_{n+1}f\|= \|\mathcal{Q}_{n+1}f\|$ follows from \eqref{3.K}.
By (iii), $\|\hat{K}_{n}^0f\|\le c \|\mathcal{P}_{n+1}f\|$  and we obtain
\begin{equation}\label{4.iv1}
\|\hat{K}_{n}f\|^2\le c^2 \|\mathcal{P}_{n+1}f\|^2+\|\mathcal{Q}_{n+1}f\|^2=
(c^2-1)\|\mathcal{P}_{n+1}f\|^2+\|f\|^2.
\end{equation}
(We use here that $\|\mathcal{P}_{n+1}f\|^2+\|\mathcal{Q}_{n+1}f\|^2=\|f\|^2$.) By the condition of part (iv)
\[
\|f\|^2(1-\epsilon)\le (c^2-1)\|\mathcal{P}_{n+1}f\|^2+\|f\|^2 \text{ and hence } \|\mathcal{P}_{n+1}f\|^2\le \frac{\epsilon}{1-c^2}\|f\|^2.
\]
So $\|\mathcal{Q}_{n+1}f\|^2=\|f\|^2-\|\mathcal{P}_{n+1}f\|^2\ge \left(1-\frac{\epsilon}{1-c^2}\right)\|f\|^2$.
\end{proof}

\subsubsection{Proof of the main technical result}\label{Sec4.2.2} In this section, we use the following simplified
notation: $\hat{V}_{\mathfrak{g}_{n}}=\hat{V}_{n}$. Theorem \ref{ThmTechn} follows from the following statement.
\begin{theorem}\label{ThmTecn} Suppose that Condition I (inequality \eqref{MainCond}) is satisfied
and that
\begin{equation}\label{MainTecn1}
\sup_n\|\hat{K}_{n-1}\hat{V}_{n}\hat{K}_{n}\|=1.
\end{equation}
Then there is a measure $\nu\in M$ and a probability measure $\kappa$ on $\mathcal{S}$ which is preserved by
the group $G_{\nu}$.
\end{theorem}
\begin{remark}\label{RmkMain} If matrices $g_j$ are independent then $K_{n-1}V_{n}=W_{\nu_n}$ and
Lemma \ref{Thm2} states that $\|W_{\nu_n}\|<1$. In contrast, $\|\hat{K}_{n-1}\hat{V}_{n}\|=1$. To see that, set
$\mathfrak{f}=\hat{V}_{n}^{-1}\mathfrak{h}$, where $\mathfrak{h}\in \mathcal{L}_{n}$ and is arbitrary otherwise.
Then $\hat{K}_{n-1}\hat{V}_{n}\mathfrak{f}=\hat{K}_{n-1}\mathfrak{h}$ and hence
$\|\hat{K}_{n-1}\hat{V}_{n}\mathfrak{f}\|= \|\hat{K}_{n-1}\mathfrak{h}\|=\|\mathfrak{h}\|= \|\mathfrak{f}\|$
which proves the claim.
\end{remark}
\begin{proof}[Proof of Theorem \ref{ThmTecn}.]

Equality \eqref{MainTecn1} implies that for a given (small) $\varepsilon>0$ there is $n(\varepsilon)$ and a
function $\varphi_{\varepsilon}\in \mathbb{H}_{n+1}$ such that
\begin{equation}\label{Main2}
\|\hat{K}_{n(\varepsilon)-1}\hat{V}_{n(\varepsilon)}\hat{K}_{n(\varepsilon)}
\varphi_{\varepsilon}\|^2\ge (1-\varepsilon)\|\varphi_{\varepsilon}\|^2.
\end{equation}
Since throughout this proof $\varepsilon$ will be fixed, we shall from now on
write $\varphi$ for $\varphi_\varepsilon$ and $n$ for $n(\varepsilon)$.

Set $\psi= \hat{K}_{n}\varphi$. Since  $\|\hat{K}_{n-1}\|=\|\hat{V}_{n}\|=\|\hat{K}_{n}\|=1$,
it follows from \eqref{Main2} that
\begin{equation}\label{Main3}
\|\psi\|^2= \|\hat{K}_{n}\varphi\|^2\ge (1-\varepsilon)\|\varphi\|^2 .
\end{equation}
(Otherwise, we would have had
$\|\hat{K}_{n-1}\hat{V}_{n}\hat{K}_{n}\varphi\|^2\le \|\hat{K}_{n}\varphi\|^2 <(1-\varepsilon)\|\varphi\|^2.$)

Similarly, and using that $\hat{V}_{n}$ is a unitary operator, we claim that
\begin{equation}\label{Main8}
\|\hat{K}_{n-1}\hat{V}_{n}\psi\|^2\ge (1-\varepsilon)\|\hat{V}_{n}\psi\|^2=(1-\varepsilon)\|\psi\|^.
\end{equation}
To see this, suppose that to the contrary
$\|\hat{K}_{n-1}\hat{V}_{n}\psi\|^2< (1-\varepsilon)\|\psi\|^2$. Then
\[
(1-\varepsilon)\|\psi\|^2>\|\hat{K}_{n-1}\hat{V}_{n}\psi\|^2=
\|\hat{K}_{n-1}\hat{V}_{n}\hat{K}_{n}\varphi\|^2 \ge (1-\varepsilon)\|\varphi\|^2
\]
and hence $\|\psi\|^2>\|\varphi\|^2$ which contradicts $\|\psi\|^2\le\|\varphi\|^2 $ and thus proves \eqref{Main8}.

By Lemma \ref{Lemma<c}(iv), the inequality in \eqref{Main3} implies that
\begin{equation}\label{4.iva}
\|\mathcal{P}_{n+1}\varphi\|^2\le \bar{\varepsilon}\|\varphi\|^2\ \text{ and }\
\|\mathcal{Q}_{n+1}\varphi\|^2\ge \left(1-\bar{\varepsilon}\right)\|\varphi\|^2,
\end{equation}
where $\bar{\varepsilon}=\frac{\varepsilon}{1-c^2}$. It follows then that
\[
\begin{aligned}
&\|\hat{K}_{n}^0\varphi\|^2=\|\hat{K}_{n}\mathcal{P}_{n+1}\varphi\|^2\le \bar{\varepsilon}\|\varphi\|^2
\ \text{ and }\\
&\|\hat{K}_{n}^1\varphi\|^2=\|\hat{K}_{n}\mathcal{Q}_{n+1}\varphi\|^2= \|\mathcal{Q}_{n+1}\varphi\|^2\ge(1-\bar{\varepsilon})\|\varphi\|^2.
\end{aligned}
\]

Similarly, by Lemma \ref{Lemma<c}(iv), the inequality in \eqref{Main8} implies that
\begin{equation}\label{Main9}
\|\mathcal{Q}_{n}\hat{V}_{n}\psi\|^2\ge (1-\bar{\varepsilon})\|\hat{V}_{n}\psi\|^2
=(1-\bar{\varepsilon})\|\psi\|^2\ge (1-\bar{\varepsilon})(1-\varepsilon)\|\varphi\|^2,
\end{equation}
where the last step is due to \eqref{Main3}. It follows from \eqref{Main9} that
\begin{equation}\label{Main9a}
\|\mathcal{Q}_{n}\hat{V}_{n}\psi\|\ge (1-\tilde{\varepsilon})\|\varphi\|,
\end{equation}
where $\tilde{\varepsilon}=(\bar{\varepsilon}+\varepsilon)/2.$
Next, since $\psi=\hat{K}_{n}^1\varphi+\hat{K}_{n}^0\varphi$,
\begin{equation}\label{Main6}
\|\mathcal{Q}_{n}\hat{V}_{n}\psi\|\le \|\mathcal{Q}_{n}\hat{V}_{n}\hat{K}_{n}^1\varphi\|+
\|\mathcal{Q}_{n}\hat{V}_{n}\hat{K}_{n}^0\varphi\|\le \|\mathcal{Q}_{n}\hat{V}_{n}\hat{K}_{n}^1\varphi\|+
{\bar{\varepsilon}}^{\frac{1}{2}}\|\varphi\|.
\end{equation}
Combining \eqref{Main9a} and \eqref{Main6} we obtain:
\begin{equation}\label{Main7}
\|\mathcal{Q}_{n}\hat{V}_{n}\hat{K}_{n}^1\varphi\|\ge (1-\tilde{\varepsilon}-{\bar{\varepsilon}}^{\frac{1}{2}})\|\varphi\|.
\end{equation}
Set $h=\hat{K}_{n}^1\varphi=\hat{K}_{n}\mathcal{Q}_{n+1}\varphi $.  Note that $h\in \mathcal{L}_{n}$ since
$\mathcal{Q}_{n+1}\varphi\in\mathcal{L}_{n+1}$ and moreover
$\|h\|=\|\mathcal{Q}_{n+1}\varphi\|\ge(1-\bar{\varepsilon})^{\frac{1}{2}} \|\varphi\|$ (see Remark \ref{remark}).
 It follows from \eqref{Main7} that
\begin{equation}\label{Main1}
\|\mathcal{Q}_{n}\hat{V}_{n}h\|\ge (1-\tilde{\varepsilon}-{\bar{\varepsilon}}^{\frac{1}{2}})(1-\bar{\varepsilon})^{-\frac{1}{2}}\|h\|.
\end{equation}
By \eqref{Def:projection1} and \eqref{3.V}
\begin{equation}\label{Main10}
\begin{aligned}
\left(\mathcal{Q}_{n}\hat{V}_{n}h\right)(u)&=\int_X \left(\hat{V}_{n}h\right)(x,u)\mu_n(dx)=
\int_X h(\mathfrak{g}_n(x).u)\|\mathfrak{g}_n(x)u\|^{-\frac m2}\mu_n(dx)\\
&=\int_{\textrm{SL}(m,\mathbb{R)}} h(g.u)\|gu\|^{-\frac m2}\nu_n(dg),
\end{aligned}
\end{equation}
where the last equality is due to the definition of $\nu_n$. Equation \eqref{Main10} shows that the
action of $\mathcal{Q}_{n}\hat{V}_{n}$ on $\mathcal{L}_n$ is isomorphic to the action of $W_{\nu_n}$
on $L_2(\mathcal{S})$ (see \eqref{M10}) and, in particular, $\|\mathcal{Q}_{n}\hat{V}_{n}\|=\|W_{\nu_n}\|$.

Since $\varepsilon$ in \eqref{Main1} can be made arbitrarily small, it follows that $\sup_{n}\|\mathcal{Q}_{n}\hat{V}_{n}\|=1$
and therefore also $\sup_{\nu\in M}\|W_{\nu_n}\|=1$. This, by Lemma \ref{Thm2} (and \ref{LemmaMain}), implies
the existence of a $\nu\in M$ and a $\kappa$ on $\mathcal{S}$ preserved by $G_\nu$.
Theorem \ref{ThmTecn} is proved.
\end{proof}

\section{Proof of Theorem \ref{Thm4}}\label{SecThm4}
Throughout this section, we suppose that the Markov chain $\xi$, the corresponding operators $K_{n}^0$,
and the sequence of functions $(\mathfrak{g}_n)_{n\ge1}$ are those defined in section \ref{SecMR}.

Our plan is as follows. We first prove Lemma \ref{LemmaParticCase} which is a particular case of Theorem \ref{Thm4}.
and then derive Theorem \ref{Thm4} from this Lemma.

As will be seen right now, Lemma \ref{LemmaParticCase}
results from easy analysis of the proof of Theorem \ref{Thm1}.

\begin{lemma}\label{LemmaParticCase} Suppose that:\newline
$\mathrm{(i)}$ the inequalities $\|K_{n}^0\|\le c$, where $ c<1$, are satisfied for all $n\ge 1$;\newline
$\mathrm{(ii)}$  all distributions $\nu_{2j},\ j\ge1$ belong to a weakly compact set $M$ of distributions
satisfying \eqref{MainCond2}(b).

Then there is a (non-random) $\lambda>0$ such that with probability 1
\begin{equation}\label{M220}
\liminf_{j\to\infty}\frac1j \ln\|g_{j}\ldots g_1\|\ge \lambda
\end{equation}
and the estimate \eqref{M220} does not depend on the choice of the subsequence of functions $(\mathfrak{g}_{2j-1})_{j\ge1}.$
\end{lemma}
\begin{proof}
An easy examination of the proof of Theorem \ref{Thm1} shows that this proof follows
from the fact that inequalities
\begin{equation}\label{M221}
\|\hat{K}_{2j-1}\hat{V}_{\mathfrak{g}_{2j}}\hat{K}_{2j}\hat{V}_{\mathfrak{g}_{2j+1}}\|
=\|\hat{K}_{2j-1}\hat{V}_{\mathfrak{g}_{2j}}\hat{K}_{2j}\|\le\alpha<1
\end{equation}
hold for all $j\ge1$.

Theorem \ref{ThmTechn}  states that \eqref{MainTechn2a} holds, with $n=2j$
if $\|\hat{K}_n^0\|\le c<1$, $\|\hat{K}_{n+1}^0\|\le c<1$, and
the group $G_{\nu_n}$ does not preserve any probability measure on $\mathcal{S}$.
The conditions of our Corollary thus imply that \eqref{M220} follows from Theorem \ref{Thm1}.

The uniformity of \eqref{M220} with respect to the choice of the subsequence of
functions $(\mathfrak{g}_{2j-1})_{j\ge1}$ from the fact that the matrices $\mathfrak{g}_{2j-1}(\xi_{2j-1})$
have no impact on the inequalities \eqref{product4} defining the value of $\lambda$. \end{proof}

Let us make one final observation concerning the proof of Theorem \ref{Thm1}: it is not important for this proof
that the phase space $X$ of the chain $\xi$ does not depend on $n$.
Indeed, the definition $\xi=(\xi_j\in X^{(j)})_{n\ge1}$, where $X^{(j)}$ is a sequence of phase spaces,
is equivalent to the original definition with $X =\cup_{j=1}^\infty X^{(j)}$ - the union
of $ X^{(j)}$'s (which are considered as disjoint sets).
Formally speaking, we also have to extend to $X$ the definitions of functions $\mathfrak{g}_j:X^{(j)}\mapsto\SL$.
This can be done e.g. by setting $\mathfrak{g}_j(y)=g\in\SL$ for all $y\not\in X^{(j)}$; the choice of the $g$
plays no role since the chain $\xi$ at time $j$ can take values only in $X^{(j)}$.

We thus shall assume that Theorem \ref{Thm1} and thus also Lemma \ref{LemmaParticCase} work for chains
with phase spaces which depend on time.

\begin{proof}[Proof of Theorem \ref{Thm4}] Set $n_0=l_0=0$ and define a sequence of intervals $[a_j,b_j],\ j\ge1,$ by
\[
\begin{aligned}
& a_{j}=n_{j-1}+l_{j-1}+1, & b_{j}=n_{j+1}, & \text{ if $j$ is odd,}\\
& a_{j}=n_{j}+1, & b_{j}=n_{j}+l_{j} & \text{ if $j$ is even}.
\end{aligned}
\]
Define a new Markov chain $\eta=(\eta_j)_{j\ge1}$ by setting $\eta_j=(\xi_{a_j}, \xi_{a_j+1},...,\xi_{b_j})$
with phase spaces $X^{(j)}=\overset{b_j-a_j+1\ \mathrm{times}}{\overbrace{X\times...\times X}}$ depending on $j$.

Define also a new sequence of functions $\bar{\mathfrak{g}}_j:X^{(j)}\mapsto\SL$ by
\[
\bar{\mathfrak{g}}_j (\eta_j)=\mathfrak{g}_{b_j}(\xi_{b_j})...\, \mathfrak{g}_{a_j}(\xi_{a_j}).
\]
Let $\nu_j^\eta$ be the distribution of $\bar{\mathfrak{g}}_j (\eta_j)$. Note that condition (ii) of Theorem \ref{Thm4}
and the definition of $[a_j,b_j]$ imply that if $j$ is even then $\nu_j^\eta=\nu_{n_j+1,l_j}$  and hence condition (ii)
of Lemma \ref{LemmaParticCase} is satisfied.

We shall now check that condition (i) of Lemma \ref{LemmaParticCase} is also satisfied.

Denote by $\tilde{H}_{j}$ the Hilbert space of real-valued functions square integrable with respect
to the measure on $X^{(j)}$ corresponding to the chain $\eta$. This measure is given by
\[
\mathbb{P}(\eta_{j}\in dy_1\times dy_2\times...\times dy_{r_j})=\mu_{a_j}(dy_1)k_{a_j}(y_1,dy_2)...\, k_{b_j-1}(y_{r_{j}-1},dy_{r_j}),
\]
where $r_j=b_j-a_j+1$ (this notation will be used throughout this proof).

Let $\tilde{K}_{j}:\tilde{H}_{j+1}\mapsto \tilde{H}_{j},\ j\ge 1,$
be the transition operators of the chain $\eta$. Below, we use the notation $\mathbf{x}=(x_1,...,x_{r_{j}})$ and
$\mathbf{y}=(y_1,...,y_{r_{j+1}})$ for elements of $\tilde{H}_{j}$ and $\tilde{H}_{j+1}$ respectively.

Denote the kernel of $\tilde{K}_{j}$ by $\tilde{k}_{j}(\mathbf{x},d\mathbf{y})$. Obviously,
\[
\tilde{k}_{j}(\mathbf{x},d\mathbf{y})=k_{b_j}(x_{r_j},dy_1)k_{a_{j+1}}(y_{1},dy_2)\ldots\, k_{{r_j}-1}(y_{r_{j+1}-1},dy_{r_{j+1}}).
\]
It is important that $\tilde{k}_{a_j}(\mathbf{x},d\mathbf{y})$ depends only on $x_{r_j}$ (and not other components of $\mathbf{x}$).

Denote by $\tilde{K}_{j}^0$ the restriction of $\tilde{K}_{j}$ to the subspace $\tilde{H}_{j}^0$ of functions
from $\tilde{H}_{j}$ with zero mean.
\begin{lemma}\label{A2} Suppose that (as in Theorem \ref{Thm4})
the inequalities $\|K_{n_j}^0\|\le c$, $\|K_{n_j+l_j}^0\|\le c$, where $ c<1$, hold for all $j\ge 1$.
Then $\|\tilde{K}_{j}^0\|\le c$ for all $j\ge1$.
\end{lemma}
\begin{proof} Throughout this proof $j$ and $r_{j+1}$ are fixed and so we write $r$ for $r_{j+1}$.

Let $f\in H_{j+1}^0$, that is
\begin{equation}\label{A3}
\int_{X^{(j+1)}} \mu_{a_{j+1}}(dy_1)k_{a_{j+1}}(y_{1},dy_2)\ldots\, k_{b_{j+1}-1}(y_{r-1},dy_{r})f(\mathbf{y})
=\int_{X} \mu_{a_{j+1}}(dy_1)\varphi(y_1)=0,
\end{equation}
where
\begin{equation}\label{Eq6.2}
\varphi(y_1)=\int_{y_2\in X,...,\,y_{r}\in X} k_{a_{j+1}}(y_{1},dy_2)\ldots\,
k_{b_{j+1}-1}(y_{r-1},dy_{r})f(y_1,y_2,..., y_{r}).
\end{equation}
By the definition of the action of $\tilde{K}_{j}$,
\begin{equation}\label{Eq6.1}
\begin{aligned}
(\tilde{K}_{j}^0f)(\mathbf{x})&=
\int_{X^{(j+1)}}k_{b_{j}}(x_{r_j},dy_1)k_{a_{j+1}}(y_1,dy_2)...\,k_{b_{j+1}-1}(y_{r-1},dy_{r})
 f(\mathbf{y})\\
 &=\int_{X}k_{b_{j}}(x_{r_j},dy_1)\varphi(y_1).
 \end{aligned}
\end{equation}
We have to show that
\begin{equation}\label{A4}
\|\tilde{K}_{j}^0f\|_{\tilde{H}_j}
\le c^2\|f\|_{\tilde{H}_{j+1}}.
\end{equation}
It follows from \eqref{Eq6.1} that
\[
\|\tilde{K}_{j}^0f\|_{\tilde{H}_{j}}^2=\int_X \left(\int_X k_{b_j}(x,dy)\varphi(y)\right)^2\mu_{b_j}(dx).
\]
On the other hand \eqref{A3} means that $\varphi\in H_{a_{j+1}}^0$ and the right hand side of \eqref{Eq6.1} coincides with
$K_{b_j}^0\varphi\in H_{b_j}^0$ and
\[
\|K_{b_j}^0\varphi\|_{H_{b_j}}^2=\int_X \left(\int_X k_{b_j}(x,dy)\varphi(y)\right)^2\mu_{b_j}(dx).
\]
Note also that $b_j$'s are defined so that  $\|K_{b_j}\|\le c$ for all $j\ge1$ by the condition of the Lemma. Therefore
\begin{equation}\label{A5}
\|\tilde{K}_{j}^0f\|_{\tilde{H}_{j}}^2=\|K_{b_j}^0\varphi\|_{H_{b_j}}^2\le c^2\|\varphi\|_{H_{b_j+1}}^2=
c^2\int_X \mu_{a_{j+1}}(dx)\varphi(x)^2,
\end{equation}
(where the last equality is due to $b_j+1=a_{j+1}$).
It follows from \eqref{Eq6.2} that, by the Cauchy-Schwartz inequality,
\[
\varphi(y_1)^2\le \int_{y_2\in X,...,\,y_{r}\in X} k_{a_{j+1}}(y_1,dy_2)\ldots\,
k_{b_{j+1}-1}(y_{r-1},dy_{r})[f(y_1,y_2,..., y_{r})]^2.
\]
The integral in the right hand side of \eqref{A5} is now estimated by
\[
\begin{aligned}
&\int_X \mu_{a_{j+1}}(dx)\varphi(x)^2 \\
&\le\int_X \mu_{a_{j+1}}(dy_1)
\int_{y_2\in X,...,\,y_{r}\in X} k_{a_{j+1}}(y_1,dy_2)\ldots\,
k_{b_{j+1}-1}(y_{r-1},dy_{r})[f(y_1,y_2,..., y_{r})]^2\\
&=\int_{ X^{(j+1)}} \mu_{a_{j+1}}(dy_1)k_{a_{j+1}}(y_1,dy_2)\ldots\,
k_{b_{j+1}-1}(y_{r-1},dy_{r})[f(y_1,y_2,..., y_{r})]^2=\|f\|_{\tilde{H}_{j+1}}^2.
\end{aligned}
\]
The Lemma is proved.
\end{proof}
We thus have shown that if the assumptions of Theorem \ref{Thm4} are satisfied then
the chain $\eta$ satisfies all assumptions of Lemma \ref{LemmaParticCase} and hence
there is a non-random $\lambda>0$ such that with probability 1
\begin{equation}\label{Eq}
\liminf_{j\to\infty}\frac1j \ln\|\bar{\mathfrak{g}}_{j}(\eta_j)\ldots \bar{\mathfrak{g}}_{1}(\eta_1)\|\ge \lambda.
\end{equation}
Since $\bar{\mathfrak{g}}_{j}(\eta_j)\ldots \bar{\mathfrak{g}}_{1}(\eta_1)=g_{n_j+l_j}\ldots g_1$, where the product
in the right side of this equality is the same as in \eqref{M22}, Theorem \ref{Thm4} is proved.
\end{proof}

\section{Appendix}\label{appendix}
The statements proved in this Appendix are simple and neither is new. We prove them in
order to make this paper more self-contained and also because in some cases it is
easier to prove the statement than to find the corresponding reference.

\subsection{Proof of \eqref{Def:projection}}\label{sec5.1}

First,we have to show that $\mathcal{P}_n$ defined by the rhs of \eqref{Def:projection} maps
$\mathbb{H}_n$ into $\mathbb{H}_n^0$. Let $h\in \mathcal{L}_{n}$. Then
\begin{align*}
\left<\mathcal{P}_nf,h\right>_{\mathbb{H}_n}&=\int_{X\times\mathcal{S}}f(x,u)h(u)\mu_n(dx)du-
\int_{X\times\mathcal{S}}\left(\int_{X}f(y,u)\mu_n(dy)\right)h(u)\mu_n(dx)du\\
&=\int_{X\times\mathcal{S}}f(x,u)h(u)\mu_n(dx)du-
\int_{\mathcal{S}}\left(\int_{X}f(y,u)\mu_n(dy)\right)h(u)du =0.
\end{align*}
It remains to check that if $f\in \mathbb{H}_n^0$ then the integral in the rhs of \eqref{Def:projection}
is vanishing. This follows from the following lemma.
\begin{lemma}\label{A1} If $f\in \mathbb{H}_n^0$ then
$a(u)\equiv \int_X f(x,u)\mu_n(dx)=0$ for Lebesgue - a.e. $u\in\mathcal{S}$.
\end{lemma}
\begin{proof}
Note first that $a\in \mathcal{L}_{n}$ because $\int_\mathcal{S}a(u)^2du\le
\int_\mathcal{S}\left( \int_X f(x,u)^2\mu_n(dx) \right)du= \|f\|_{\mathbb{H}_n}^2$.
By the definition of $\mathbb{H}_{n}^0$, $\int_{X\times\mathcal{S}}f(x,u)a(u)\mu_{n+1}(dx)du=0$.
On the other hand,
\[
\int_{X\times\mathcal{S}}f(x,u)a(u)\mu_{n+1}(dx)du=
\int_{\mathcal{S}}\left(\int_X f(x,u)\mu_{n+1}(dx)\right)a(u)du=\int_{\mathcal{S}}a(u)^2du.
\]
The Lemma is proved. \end{proof}

\subsection{Proof of the equivalence of \eqref{M21} and \eqref{M13}} \label{proofM13}

We have to check that for any $g\in \textrm{SL}(m,\mathbb{R})$
\begin{equation}\label{A13}
\int_{\mathcal{S}}\psi(u)(g^{-1}\kappa)(du)=\int_{\mathcal{S}}\psi(g^{-1}.u)\kappa(du).
\end{equation}
Obviously, it suffices to check \eqref{A13} for characteristic functions of Borel subsets of $\mathcal{S}$.
Let $A\subset \mathcal{S}$ be such a subset and $\chi_A$ be its characteristic function. Then
\[
\int_{\mathcal{S}}\chi_A(u)(g^{-1}\kappa)(du)=(g^{-1}\kappa)(A)=\kappa(g.A).
\]
Since $\chi_A(g^{-1}.u)=\chi_{g.A}(u)$, we have
\[
\int_{\mathcal{S}}\chi_A(g^{-1}.u)\kappa(du)=\int_{\mathcal{S}}\chi_{g.A}(u)\kappa(du)=\kappa(g.A)
\]
and this proves \eqref{A13}.

\end{document}